%% file: Template_1.tex
\documentclass{aims}
\usepackage{amsmath}
  \usepackage{paralist}
  \usepackage{graphics} 
  \usepackage{epsfig} 
\usepackage{graphicx}  \usepackage{epstopdf}
 \usepackage[colorlinks=true]{hyperref}
\hypersetup{urlcolor=blue, citecolor=red}

\def\currentvolume{X}
 \def\currentissue{X}
  \def\currentyear{200X}
   \def\currentmonth{XX}
    \def\ppages{X--XX}

\newtheorem{theorem}{Theorem}[section]
\newtheorem{corollary}{Corollary}

\newtheorem{lemma}[theorem]{Lemma}
\newtheorem{proposition}{Proposition}

\theoremstyle{definition}
\newtheorem{definition}[theorem]{Definition}
\newtheorem{remark}{Remark}

\usepackage{tikz}
\usetikzlibrary{calc}
\usetikzlibrary{decorations.pathreplacing,angles,quotes,decorations.markings}

\usepackage{import}
\usepackage{xifthen}
\usepackage{pdfpages}

\newcommand{\Z}{\ensuremath{\mathbb{Z}}}
\newcommand{\R}{\ensuremath{\mathbb{R}}}
\newcommand{\B}{\ensuremath{\mathcal{B}}}
\newcommand{\T}{\ensuremath{\mathcal{T}}}
\newcommand{\dd}{\mathop{}\!\mathrm{d}}
\DeclareMathOperator*{\argmin}{arg\,min}

\title[Frenkel-Kontorova models on quasi-crystals] 
      {On the existence of solutions for the Frenkel-Kontorova models on quasi-crystals}

\author[Jianxing Du and Xifeng Su]{}

 \email{jxdu000@gmail.com}
 \email{xfsu@bnu.edu.cn}

\begin{document}
\maketitle

\centerline{\scshape Jianxing Du and Xifeng Su}
\medskip
{\footnotesize

} 

\bigskip


\begin{abstract}
  This article  focuses on recent investigations on equilibria of the Frenkel-Kontorova models subjected to potentials generated by quasi-crystals. 

We present a specific one-dimensional model with an explicit potential driven by the Fibonacci quasi-crystal. For a given positive number $\theta$, we show that there are multiple equilibria with rotation number $\theta$, e.g., a minimal configuration and a non-minimal equilibrium configuration. Some numerical experiments verifying the existence of such equilibria are provided. 
\end{abstract}

\section{Introduction}
We consider the Frenkel-Kontorova models with \emph{quasi-periodic} potentials. One may refer to \cite{BK2004} for several physical interpretations of such models.

In order to give a unified picture of the known results, let's take the one-dimensional quasi-crystals case for example (see \cite{Sadun2008} for more advanced studies on general quasi-crystals). Frenkel-Kontorova models describe the dislocations of particles deposited over a substratum given by the quasi-crystals (see \cite{BK2004}). That is , we model the position of the $i$-th particle by $x_i\in\R$ and the total energy is formally composed of the summation over all $i\in \mathbb{Z}$ of spring potential $\frac{1}{2} (x_{i}-x_{i+1})^2$ between the nearest neighbors $i$-th and the $(i+1)$-th particles,
and the potential $V(x_i)$ from interaction of the $i$-th particle with the quasi-periodic substratum:
 \[
 \mathcal{S}((x_i)_{i\in \mathbb{Z}}):=\sum_{i\in\Z} \left[ \frac{1}{2} (x_{i}-x_{i+1})^2+V(x_i) \right].
 \] 
In this article, we are concerned with the case that the potential function $V$ is a pattern equivariant function (see Definition~\ref{pattern equivariant functions} for details) and aim to discuss existing approaches to prove the existence of equilibria. 
Related problems of finding equilibria generated by other types of quasi-periodic functions are discussed in \cite{SuL2012, SuL2017} and references therein. \cite{Lions2003} presents an example to  show the non-existence of bounded correctors in the setting of homogenization.

The aperiodicity will bring us difficulties such as loss of compactness and less results are known for quasi-periodic than for periodic. We now recall the existing results in the literature.
In \cite{Aubry1990,Trevino2019}, under the assumptions that the potential function is large enough, 
the authors use the idea of the anti-integrable limits to obtain multiple equilibria with any prescribed rotation number or without any rotation numbers. This approach also works for higher dimensional generalization of the quasi-periodic Frenkel-Kontorova model but requires that the corresponding system is  far away from the integrable case. 

In another direction, \cite{GGP2006} considers the one-dimensional Fibonacci quasi-crystals without extra assumption on whether the system is integrable or not, and  mainly uses topological methods to establish the existence of the minimal configurations with any given rotation number. Here it is rather essential that the configuration space is one dimension. Note that the minimal configuration is a special equilibrium with extra properties.
One may refer to \cite{GPT2017} for other methods to search minimal configurations.

In this article, we construct an explicit potential $V$ driven by the Fibonacci quasi-crystals.
Our goal here is to provide several approaches to tackle variational problem for the Frenkel-Kontorova models. More precisely, we aim to find minimal configurations and non-minimal equilibrium configurations of the Frenkel-Kontorova model with quasi-periodic potentials.  

In order to give a simple and self-contained proof, we fix at the beginning a positive number $\theta=(3\tau+1)/2$ where $\tau=(\sqrt{5}+1)/2 =1.618\cdots$.
We state the main result as follows:
\begin{theorem}\label{maintheorem}
	For the Frenkel-Kontorova model with the potential $V$ (which is defined in Section \ref{EP}),  
	\begin{enumerate}[(i)]
		\item 	there exists a minimal configuration with the rotation number $\theta$, 
		\item   there exists a non-minimal equilibrium configuration  with the rotation number $\theta$.
	\end{enumerate}
\end{theorem}

\begin{remark}
\begin{itemize}
	\item [(1)] We choose $\theta=(3\tau+1)/2$ for computational convenience. In fact, Theorem \ref{maintheorem} is valid for any $\theta\in\R$.
	\item [(2)] Related results could also hold for other one-dimensional quasi-periodic tilings.
	\end{itemize}
\end{remark}

But so far, analogues of the above result in higher dimensions  are not yet known. 
We hope to get some inspiration from the specific example and corresponding numerical simulations. For numerical computations of minimal configurations, one may refer to \cite{Llave2013}.

\subsection*{Organization of the article}
In Section~\ref{pre}, we  introduce some necessary fundamentals about quasi-crystals and the variational problem for the Frenkel-Kontorova models. In particular, we  give an example of the quasi-crystals, the Fibonacci chain, and an example of the Frenkel-Kontorova models with quasi-periodic potentials. 

Section~\ref{solutions} is devoted to the proof of  Theorem~\ref{maintheorem}. Numerical simulations searching for  equilibrium configurations are provided here.

In fact, in Section~\ref{sec : minimal configurations}, we obtain the minimal configurations for the Frenkel-Kontorova models on one-dimensional Fibonacci quasi-crystals with slight modifications compared with those in \cite{GGP2006}. Thus we finish the proof of item (i) in Theorem~\ref{maintheorem}. Of course, the KAM circles are minimizers.

We  apply the idea of anti-integrable limits to prove the existence of equilibrium configurations of type $h$ in Section~\ref{existence of equi conf}. Then we  prove that these equilibrium configurations are non-minimal, which completes the proof of item (ii) in Theorem~\ref{maintheorem}.

\section{Preliminaries}\label{pre}
We recall several standard notions of quasi-crystals in Section~\ref{notion: quasi-crystal}.  In Section~\ref{VP}, we introduce the minimal and equilibrium configurations of  the variational problem for the Frenkel-Kontorova models. Particularly, in Section~\ref{EP},  we state the definition of pattern equivariant potential which we take as the potential in the variational problem, and finally an example is provided.

\subsection{Quasi-crystals}\label{notion: quasi-crystal}
In the $d$-dimensional Euclidean space  $\R^d$,  the  \emph{open ball} centered at $x$ with radius $r$ is denoted by $B_r(x)$. The \emph{closure} of a set $A$ and  its \emph{cardinality}  are denoted by  $\overline{A}$ and $\operatorname{Card}(A)$, respectively. The \emph{Lebesgue measure} of a Lebesgue-measurable set  $A$ is denoted by $\lambda(A)$.  Countable subsets in $\R^d$ are called \emph{point sets}. 

A point set $\Lambda\subset \R^d$ is \emph{uniformly discrete} if there exists  $r>0$ such that $(x+B_r(0))\cap(y+B_r(0))=\emptyset$ holds for all distinct $x,y\in\Lambda$, where the sign ``$+$'' is the \emph{Minkowski sum}. A point set $\Lambda$ is \emph{relative dense} if there exists $R>0$ such that $\Lambda+\overline{B_R(0)}=\R^d$.

\begin{definition}[Delone sets]
A point set $\Lambda\subset \R^d$ is  \emph{Delone} if it is both uniformly discrete and relative dense.
\end{definition}

A \emph{cluster} of the point set $\Lambda$ is the intersection $K\cap \Lambda$ for some compact set $K\subset\R^d$. In particular, if the compact set $K$ is convex, we call such cluster $K\cap \Lambda$ a \emph{patch}. Two clusters $P_1$ and $P_2$  are said to be  \emph{equivalent} if there exists a vector $v\in\R^d$ such that $P_1+v=P_2$.


\begin{definition}[finite local complexity]
A point set $\Lambda\subset \R^d$ is of  \emph{finite local complexity} if for any $M>0$, the point set $\Lambda$ possesses only finitely many equivalence classes of clusters with diameters smaller than $M$, where the diameter of a subset $A\subset \R^d$ is $\sup_{x,y\in A}|x-y|.$ 
\end{definition}

\begin{definition}[repetitive]
	A point set $\Lambda\subset \R^d$ is \emph{repetitive} if for any cluster $P\subset \Lambda$, there exists $R>0$ such that  any ball with radius $R$ contains a cluster equivalent to $P$.
\end{definition}

\begin{definition}[non-periodic]
A point set $\Lambda\subset\R^d$ is \emph{non-periodic} if  $t+\Lambda\ne \Lambda$ for any $t\in\R^d\backslash\{0\}$.
\end{definition}

\begin{definition}[quasi-crystal]
 A point set $\Lambda\subset\R^d$ is called a \emph{quasi-crystal} if it is repetitive,  non-periodic and of finite local complexity.
\end{definition}

Let $K\subset\R^d$ be a compact set and fix a cluster $P=\Lambda\cap K$ of a point set $\Lambda$. Given a ball $B_r(a)$, we consider the quantity
\[\frac{\operatorname{Card} \{t\in B_r(a)\mid (-t+\Lambda)\cap K=P\}}{\lambda(B_r(a))}.\]
If the above quantity converges as $r\rightarrow\infty$ with $a\in\R^d$ fixed, we call the limit the \emph{absolute frequency} of the cluster $P$ in $\Lambda$ at $a$  and denote it by $\operatorname{Freq}_a(P)$. If the convergence is uniform in $a$, the limit is called the \emph{uniform absolute frequency} of $P$ and is denoted by $\operatorname{Freq}(P)$.

 \subsubsection{The Fibonacci chain as a specific one-dimensional quasi-crystal}\label{sec:fibonaccichain}
In this section, we will provide the Fibonacci chain as an example of one-dimensional quasi-crystals whose construction will be divided into the following three steps.

\textbf{Step 1. Construct  a one-sided Fibonacci word.} Consider a two-letter alphabet $\{a,b\}$ and the free group $\langle a,b\rangle$ generated by letters $a$ and $b$. A \emph{substitution rule} $\rho$ on $\{a,b\}$ is an endomorphism of $\langle a,b\rangle$. Specifically, we consider the substitution rule \[\rho:\begin{array}{cc} a\mapsto ab\\ b\mapsto a \end{array}\] and define a sequence $(u^{(i)})_{i\in\mathbb{Z}^+}$ of finite words  by 
\begin{equation}\label{definition of u^(i)}
u^{(i+1)}=\rho\left(u^{(i)}\right)\,\forall i\geq 1 \text{ with } u^{(1)}=a.
\end{equation}
The iterating process can be illustrated as follows:
$$
\begin{aligned}
a \stackrel{\rho}{\longmapsto}  a b \stackrel{\rho}{\longmapsto}  a b a \stackrel{\rho}{\longmapsto} a b a  a b  \stackrel{\rho}{\longmapsto}  abaababa
\stackrel{\rho}{\longmapsto} abaababaabaab \stackrel{\rho}{\longmapsto} \cdots.
\end{aligned}
$$

It can be directly derived from the definition by induction that  
\begin{equation}\label{eq: induction of u}
u^{(i+2)}=u^{(i+1)}u^{(i)} \text{  for all } i\geq 1.
\end{equation}
Thus the sequence $(u^{(i)})_{i\in\mathbb{Z}^+}$ converges  to an infinite word 
\[u:=u_0u_1u_2\cdots:=abaababaabaab\cdots\]
as $i\rightarrow +\infty$ in the product topology of $\{a,b\}^{\mathbb{Z}^+}$, where $u_0, u_1, \dots, u_j, \dots\in\{a,b\}$. 
The limit $u$ is a fixed point of $\rho$ and is called the \emph{one-sided Fibonacci word}. 

\begin{remark}	
Notice that the length of $u^{(i)}$ is the \emph{Fibonacci number} $f_i$ defined by $f_1=1$, $f_2=2$ and $f_{i+2}=f_{i+1}+f_{i}$ for all $i\geq 1$. As a convention, let $f_0=1, f_{-1}=0$ and $\tau=(1+\sqrt{5})/{2}$. It is easy to show that 
\begin{equation}\label{general term formula of Fibonacci}
f_i=\frac{\tau^{i+1}-(1-\tau)^{i+1}}{\sqrt{5}}.
\end{equation}
Furthermore, the number of times the letter $a$ appears in $u^{(i)}$ is $f_{i-1}$ and the number of times the letter $b$ appears in  $u^{(i)}$ is $f_{i-2}$.
\end{remark}

\textbf{Step 2. Construct  a two-sided Fibonacci word.} Define a sequence\\ $(w^{(i)})_{i\in\mathbb{Z}^+}$ of finite words by $w^{(i)}=u^{(i)}|u^{(i)}$: 
$$
\begin{aligned}
a|a &\stackrel{\rho}{\longmapsto} \underline{a b}| a b \stackrel{\rho}{\longmapsto} a \underline{b a}|a b a \stackrel{\rho}{\longmapsto} a b a \underline{a b}| a b a a b \stackrel{\rho}{\longmapsto}  abaaba\underline{ba}|abaababa\\
&\stackrel{\rho}{\longmapsto} abaababaaba\underline{ab}|abaababaabaab \stackrel{\rho}{\longmapsto} \cdots
\end{aligned}
$$ 
where the vertical line indicates the reference point. In fact, we can get palindromes by eliminating the last two letters of $u^{(i)}$ for all $i\geq 3$, see \cite{AB2010}. 
Thus we have \[w^{(i)}=\begin{cases}
b^{-1}a^{-1}\widetilde{u^{(i)}}\underline{ba}|u^{(i)} &\text{ if } i \text{ is  odd, }\\
a^{-1}b^{-1}\widetilde{u^{(i)}}\underline{ab}|u^{(i)} &\text{ if } i \text{ is  even, }
\end{cases}\] for all $i\geq 3$, where $\widetilde{u^{(i)}}$ denotes the reversal of $u^{(i)}$.
Furthermore $(w^{(i)})_{i\in\mathbb{Z}^+}$ has two limit points that are 2-periodic points under the substitution $\rho$: 
 $$
 \begin{aligned}
\widetilde{u}\underline{ba}|u:=&\cdots abaababaabaababaaba\underline{ba}|abaababaabaababaababa\cdots\\
\widetilde{u}\underline{ab}|u:=&\cdots abaababaabaababaaba\underline{ab}|abaababaabaababaababa\cdots.
\end{aligned}
$$

That is, the two bi-infinite words are fixed points of $\rho^2$. Here we only consider the first bi-infinite word whose underlined position is $\underline{ba}$, and  we denote it by \[w:=\cdots w_{-3}w_{-2}w_{-1}|w_{0}w_{1}w_{2}w_{3}\cdots:=\widetilde{u}\underline{ba}|u=\cdots aba\underline{ba}|abaab\cdots\] where $w_j\in\{a,b\}$ for all $j\in\Z$.  

\textbf{Step 3. Obtain a point set by the bi-infinite word as a quasi-crystal.} Let $w_{[k,l]}$  be the finite subword of $w$ from position $k$ to $l$ where $k,l\in\Z $ and $k\leq l$. Note that $w_{[k, k]}=w_{k}\in \{a, b\}$.  We then define an assignment function $|\cdot|$ on any finite word of alphabet $\{a,b\}$ by
\[
|x_1x_2|=|x_1|+|x_2|\text{ for any two finite words $x_1$ and $x_2$},
\] with $|a|=\tau$, $|b|=1$. Note that this is a cocycle in the free group.

Then we can recursively define a bi-infinite sequence $S:=(S_i)_{i\in\Z}\in\R^\Z$ by $S_i=S_{i-1}+|w_{i-1}|$ with $S_0:=0\in\R$ and $w_{i-1}\in \{a, b\}$ constructed in Step 2. 
The range of the sequence $S:\Z\to\R$,  also denoted by $S$, is called the \emph{Fibonacci chain}. 

\begin{remark}
	Our discussion will concern three kinds of words - geometric words, symbolic words and interval words. We will use the following example to illustrate the relationship among these three kinds of words. The geometric word    $S\cap[-\tau, \tau+1]$ is the point set $\{-\tau, 0,\tau, \tau+1\}$ and its corresponding symbolic word is $a|ab$ which is more visually appealing. It is obvious that geometric and symbolic words can be transformed into each other. The corresponding interval word is the left-closed and right-open interval $[-\tau,\tau+1)$. The geometric and symbolic words can be transformed into the interval words but not vice versa.
\end{remark}

\begin{lemma}\label{lemma: quasi-crystal}
The Fibonacci chain $S$ is a quasi-crystal, so we also call $S$ the Fibonacci quasi-crystal.
\end{lemma}

\begin{proof}
	We will prove the lemma by checking the definition of quasi-crystals.

(i). We will first show that $S$ is repetitive. Suppose that $P=K\cap S$ is a cluster of $S$ where $K$ is a  compact set in $\R$. Then $P$ is a finite subset of $S$ and we denote the corresponding index set by  $I = \{ i\in \mathbb{Z}  : S_i\in P  \}$. 
Consider the finite word $w^P:=w_{[\min I-1,\max I]} $ corresponding to $P$. There exists $i\in\mathbb{Z}^+$ such that $w^P$ is a subword of $w^{(i)}$. Since $w^{(1)}$ is a subword of $u^{(4)}$, by induction, $w^{(i)}$ is a subword of $u^{(i+3)}$. The word $bb$ never occurs in $w$ since $b$ only appears in $ab$ as an image of $\rho$. Thus any $2$-letter subword of $w$ contains $a$.  Any $6$-letter subword of $w$ contains two images of some letter under $\rho$ and so contains $\rho(a)$. By induction, any $(2^{i+3}-2)$-letter subword of $w$ contains $\rho^{i+2}(a)=u^{(i+3)}$ and thus contains $w^P$. Any ball of radius $2^{i+2}\tau$ in $\R$ contains at least $2^{i+3}$ points in $S$ and the corresponding $(2^{i+3}-1)$-letter subword of $w$ contains $w^P$. Hence $P$ is contained in any ball of radius $2^{i+2}\tau$.

(ii). Obviously, $S$ is of finite local complexity since $|S_i-S_{i-1}|$ ranges on $\{\tau,1\}$ for all $i\in\Z$ and there are only a limited number of arrangements of intervals in a bounded range. 

(iii). If $S$ has non-zero period, then $w=\lim\limits_{n\rightarrow +\infty} x^n$ for some finite word $x$. Then the frequency of the letter $a$ in $w$ would be the quotient of the number of letter $a$ in $x$ divided by the length of $x$, which is rational. But on the other hand, the frequency of the letter $a$ is 
\[\lim_{i\to\infty}\frac{\text{ the number of letter $a$ in }u^{(i)}}{\text{ the length of }u^{(i)}}=\lim_{i\to\infty}\frac{f_{i-1}}{f_i}=\frac{1}{\tau}\]
and is irrational, which is a contradiction.
\end{proof}
{\color{red}
}
In conclusion, we have constructed  the Fibonacci quasi-crystal $S$ which is a geometric word associated with the bi-infinite symbolic word $w$.

\subsection{The variational framework for the Frenkel-Kontorova model}\label{VP}

We consider the space $\R^\Z$ of bi-infinite real-valued sequences with the product topology. An element $x\in\R^\Z$ is denoted by $(x_i)_{i\in\Z}$ and is sometimes called a \emph{configuration}.

Given a function $H:\R^2\rightarrow\R$ we extend $H$ to arbitrary finite segments\\ $(x_j, \dots, x_k), j<k$, of configuration $x\in\R^\Z$ by \[H(x_j, \dots, x_k):=\sum_{i=j}^{k-1}H(x_i,x_{i+1}).\] We say that the segment $(x_j, \dots, x_k)$ is \emph{minimal} with respect to $H$ if \[H(x_j, \dots, x_k)\leq H(x_j^*, \dots, x_k^*)\] for all $(x_j^*, \dots, x_k^*)$ with $x_j=x_j^*$ and $x_k=x_k^*$.

\begin{definition}[minimal configuration]
	A configuration $x\in\R^\Z$ is \emph{minimal} if every finite segment of $x$ is minimal. 
\end{definition}

\begin{definition}[stationary configuration]
	If $H$ is $C^2$, we say that $x\in \R^\Z$ is \emph{a stationary configuration or an equilibrium configuration} if \begin{equation}\label{eq:equilibrium}
		\partial_2 H(x_{i-1}, x_i)+\partial_1 H(x_{i}, x_{i+1})=0\text{ for all } i\in\Z.
	\end{equation} 
\end{definition}
Obviously, each minimal configuration is a stationary configuration. 

In this article, we take the funciton $H$ as follows:
\begin{equation}\label{eq:interaction}
	H(\xi,\eta)=\frac{1}{2}(\xi-\eta)^2+V(\xi).
\end{equation}
Then (\ref{eq:equilibrium}) becomes
\[2x_i-x_{i-1}-x_{i+1}+V'(x_i)=0 \quad \text{ for all } i\in\Z.\] 
In order to further describe the  configurations, we introduce the following two notions.
\begin{definition}[rotation number]
	Let $\rho\in\R$. A configuration $x\in\R^\Z$ has a rotation number equal to $\rho$ if the limit 
	\[\lim_{i\to\pm\infty}\frac{x_i}{i}=\rho.\]
\end{definition}
\begin{definition}[type-$h$ configurations]
Let $h:\Z\rightarrow\R$. A configuration $x\in\R^\Z$ is \emph{type-$h$} if \[\sup_{i\in\Z}|x_i-h(i)|<\infty.\]
\end{definition}
 It is easy to see that the notion of the type-$h$ configurations  is more general than that of the rotation number. Taking $h(i)=\rho i$ for instance, any configuration of type-$h$ has a rotation number $\rho$.
 
\subsection{The equivariant potential}\label{EP}
In this section, we aim to build some special equivariant potential generated by the Fibonacci quasi-crystal.
We will first introduce the following notion of ``equivariant'' in our settings.
\begin{definition}\label{pattern equivariant functions}
	For any point set $\Lambda\subset\R^d$, we say a continuous function $f:\R^d\rightarrow\R$ is  $\Lambda$-\emph{equivariant}  if there exists $R>0$ such that if $x,y\in\R^d$ satisfy \[(\Lambda-x)\cap B_R(0)=(\Lambda-y)\cap B_R(0)\] then $f(x)=f(y)$.
\end{definition}

Given $x\geq 0$, we denote
\[
\alpha(x) := \max\{ y\in S ~|~ y\leq x\}, \qquad \beta(x) :=  \min\{ y\in S ~|~ y >  x\}.
\]
It is easy to see that both $\alpha$ and $\beta$ are right-continuous step functions.

For all $i\in\mathbb{Z}^+$, we can regard  $u^{(i)}$ defined in \eqref{definition of u^(i)} as a left-closed and right-open interval in $\R$, and one can always find some $i(x)\in\mathbb{Z}^+$ such that $x\in u^{(i(x))} \setminus u^{(i(x)-1)}$ with $u^{(0)}=\emptyset$. 
Or, in other words, due to \eqref{general term formula of Fibonacci}, we have 
\begin{equation}
|u^{(i)}|=f_{i-1}\tau+(f_i-f_{i-1})=\tau^i  \quad \text{ for all }i\in\mathbb{Z}^+,
\end{equation}
 and so we get either 
 \[ 0\leq x < \tau \quad \text{ or } \quad \tau^{n_1}\leq x<\tau^{n_1+1} \text{ for some } n_1\in\mathbb{Z}^+.\]
Let us consider the later case, and we have $0\leq x-\tau^{n_1}< \tau^{n_1-1}$. 
 Therefore, considering $x-\tau^{n_1}$ instead of $x$, we get either 
 \[
 0\leq x-\tau^{n_1}< \tau \quad \text{ or } \quad \tau^{n_2}\leq x-\tau^{n_1}<\tau^{n_2+1} \text{ for some }n_2\in\mathbb{Z}^+.
 \] 
 Hence, we could repeat the above procedure finitely many times and  obtain either 
(i) $0\leq x< \tau$ or (ii) there exist finitely many $n_1,n_2, \dots, n_r \in \mathbb{Z}^+$ such that
\begin{equation}\label{eq: position of x}
0\leq x-\tau^{n_1}-\tau^{n_2}-\cdots-\tau^{n_r}< \tau.
\end{equation}

This means that  $x$ is covered either by the associated closed interval of $u^{(1)}=a$ or $u^{(n_1)}u^{(n_2)}\cdots u^{(n_r)}v$ where $v$ is an unknown word. By (\ref{eq: position of x}), it would be easy to see that $v$ could be $a$, $b$, $ba$ and $bb$.

Based on the discussion above, we could obtain the following facts.
\begin{lemma}
	 If $n_r > 1$, then $v=a$. If $n_r=1$, then $v=b$. Moreover, we have \begin{align*}
    &\beta(x)=\begin{cases}
      \alpha(x)+1, &\text{ if } n_r=1\\
      \alpha(x)+\tau,  &\text{ otherwise, } 
      \end{cases} \\ 
      \text{ where }
       &\alpha(x)=\begin{cases}
      0,& \text{ if } 0\leq x< \tau \\
      \tau^{n_1}+\tau^{n_1}+\cdots+\tau^{n_r}, & \text{ if }  x\geq \tau.
      \end{cases}
   \end{align*}
	 
\end{lemma}

\begin{proof}
One could immediately have the expression of $\alpha(x)$. It then suffices to analyze the position of $x$  in the quasi-crystal rinto the following two cases(see Figure \ref{fig:positionofpoint}):
\begin{description}
\item [Case 1.] If $n_r=1$ and so $n_r+1=2$, we have either $x-\tau^{n_1}-\cdots-\tau^{n_{r-1}}-\tau^{n_r} = 0$, that is $x$ is just both the right end point of the associated interval of $a$ and the left end point of the associated interval of $b$ or $0< x-\tau^{n_1}-\cdots-\tau^{n_{r-1}}-\tau^{n_r} <\tau$, that is $x$ is inside the associated interval of $b$. In both cases, we have $\beta(x) - \alpha(x) = 1$.

\item [Case 2.] If $n_r > 1$, we obtain $u^{(n_r +1)} = u^{(n_r)} u^{(n_r -1)} $ due to \eqref{eq: induction of u}, that is, $x$ is covered by $u^{(n_1)}\cdots u^{(n_r-1)}u^{(n_r+1)}$ but not $u^{(n_1)}\cdots u^{(n_r-1)}u^{(n_r)}$.
Therefore, $v$ must be a front part of $u^{(n_r-1)}$. Then the first letter of $v$ is $a$ and $v=a$ and so $\beta(x) - \alpha(x) = \tau$.
\end{description}

\begin{figure}[htp]
\centering

\begin{tikzpicture}
\foreach \t in {1/2, 1/2+1.618, 2*1.618+3/2, 3*1.618+3/2, 4*1.618+5/2} {\draw[thick,red] (\t,0) -- (\t+1.618,0);
	\draw[thick,red] (\t, 0) -- (\t, 0.1);
}
\foreach \t in {2*1.618+3/2, 4*1.618+5/2} {\draw[thick,blue] (\t-1,0) -- (\t,0);
	\draw[thick,blue] (\t-1, 0) -- (\t-1, 0.1);
}
\draw[thick,blue] (0,0) -- (1/2,0);
\draw[thick,blue] (5*1.618+5/2,0) -- (5*1.618+3,0);
\draw[thick,blue] (5*1.618+5/2,0) -- (5*1.618+5/2,0.1);
\fill (4*1.618+2.1,0) circle (2pt);
\node[above] (x) at (4*1.618+2.1,0) {$x$};
\draw [decorate,decoration={brace,amplitude=5pt,mirror,raise=1ex}]
(1/2+1.618,0) -- (3/2+3*1.618,0) node[midway,yshift=-1.5em]{$u^{(n_{r-1})}$};	
\draw [decorate,decoration={brace,amplitude=5pt,mirror,raise=1ex}]
(3/2+3*1.618,0) -- (3/2+4*1.618,0)   node[midway,yshift=-1.5em]{$u^{(n_r)}$};
\draw [decorate,decoration={brace,amplitude=5pt,mirror,raise=1ex}]
(3/2+4*1.618,0) -- (5/2+4*1.618,0)   node[midway,yshift=-1.5em]{$v$};
\begin{scope}
\clip (0,0.0) rectangle(1/2+1.618,-2);
\draw [decorate,decoration={brace,amplitude=5pt,mirror,raise=1ex}] (-1,0) to (1/2+1.618,0) node [midway,yshift=-1.5em,xshift=2em]{$u^{(n_{r-2})}$};
\end{scope}
\end{tikzpicture}

\begin{tikzpicture}
\foreach \t in {1/2, 1/2+1.618, 2*1.618+3/2, 3*1.618+3/2, 4*1.618+5/2} {\draw[thick,red] (\t,0) -- (\t+1.618,0);
	\draw[thick,red] (\t, 0) -- (\t, 0.1);
}
\foreach \t in {2*1.618+3/2, 4*1.618+5/2} {\draw[thick,blue] (\t-1,0) -- (\t,0);
	\draw[thick,blue] (\t-1, 0) -- (\t-1, 0.1);
}
\draw[thick,blue] (0,0) -- (1/2,0);
\draw[thick,blue] (5*1.618+5/2,0) -- (5*1.618+3,0);
\draw[thick,blue] (5*1.618+5/2,0) -- (5*1.618+5/2,0.1);
\fill (3*1.618+2.1,0) circle (2pt);
\node[above] (x) at (3*1.618+2.1,0) {$x$};
\draw [decorate,decoration={brace,amplitude=5pt,mirror,raise=1ex}] (1/2+1.618,0) -- (3/2+3*1.618,0) node[midway,yshift=-1.5em]{$u^{(n_r)}$};	
\draw [decorate,decoration={brace,amplitude=5pt,mirror,raise=1ex}]
(3/2+3*1.618,0) -- (3/2+4*1.618,0)   node[midway,yshift=-1.5em]{$v$};
\begin{scope}
\clip (0,0.0) rectangle(1/2+1.618,-2);
\draw [decorate,decoration={brace,amplitude=5pt,mirror,raise=1ex}] (-1,0) to (1/2+1.618,0) node [midway,yshift=-1.5em,xshift=2em]{$u^{(n_{r-1})}$};
\end{scope}
\end{tikzpicture}
\caption{The position of a point $x$}
\label{fig:positionofpoint}
\end{figure}

\end{proof}


Notice that we have defined $\alpha(x)$ and $\beta(x)$ only for $x\geq0$. Since $w=\widetilde{u}ba|u$, we can extend $\alpha(x)$ and $\beta(x)$ on the entire real line $\R$ by \begin{align*}
  \alpha(x)&=\begin{cases}
    -\beta(-x),& \text{ if } -\tau^2\leq x< 0 \\
    -\beta(-x-\tau^2)-\tau^2, & \text{ if }  x< -\tau^2 
    \end{cases} \\ 
    \beta(x)&=\begin{cases}
      -\alpha(-x),& \text{ if } -\tau^2\leq x< 0 \\
      -\alpha(-x-\tau^2)-\tau^2, & \text{ if }  x< -\tau^2. 
      \end{cases}
\end{align*}

Let $\zeta:\R\rightarrow\R$ be a function given by 
\[\zeta(x)=
\begin{cases} 
\frac{64}{27} (3  |x|-1)^2 (96 |x| -11), & \text{if}\,x\in\left(-\frac{1}{3},-\frac{1}{4}\right)\cup\left(\frac{1}{4},\frac{1}{3}\right), \\
-64x^2+\frac{160}{27}, &  \text{if}\, x\in\left[-\frac{1}{4},\frac{1}{4}\right], 
\\
0, & \text{otherwise}.
\end{cases}
\]       

\begin{figure}[htp]
	\centering
\begin{tikzpicture}[thick, scale=3]
	\draw[thick, ->] (-1,0) -- (1,0)node  [right]{$x$};
	\draw[thick, ->] (0,0) -- (0,1.2)node  [right]{$y$};
	\draw[domain=-0.25:0.25] plot (\x, {(-64*\x*\x+(160/27))/7});
	\draw[domain=0.25:(1/3)] plot (\x, {((64/27)*(3*\x-1)*(3*\x-1)*(96*\x-11))/7});
	\draw[domain=(-1/3):-0.25] plot (\x, {((64/27)*(-3*\x-1)*(-3*\x-1)*(-96*\x-11))/7});
    \draw[dashed] (-1/4,{52/27/7})-- ({1/4}, {52/27/7});
    \draw[dashed] (1/4,{52/27/7})-- ({1/4}, {0});
    \draw[dashed] (-1/4,{52/27/7})-- ({-1/4}, {0});
    \node  at (1/3,-1/10) {$\frac{1}{3}$};
    \node  at (1/4,-1/10) {$\frac{1}{4}$};
    \node  at (0,-1/10) {$0$};
    \node  at (-1/4,-1/10) {$-\frac{1}{4}$};
    \node  at (-1/3-1/10,-1/10) {$-\frac{1}{3}$};
    \node  at (1/10,52/27/7+1/10) {$\frac{52}{27}$};
    \node  at (1/10,160/27/7+1/20) {$\frac{160}{27}$};
\end{tikzpicture}
\caption{The graph of $\zeta$}
\end{figure}

It is easy to check that $\zeta$ is $C^1$ everywhere and  $C^2$ except at $x=\pm1/3$.
The potential $V:\R\rightarrow\R$ of the interaction is defined by \begin{equation}\label{eq:potentialV}
	V(x)=
\begin{cases} 
\zeta(x-\alpha(x)), &\text{ if }2x\leq\alpha(x)+\beta(x),
\\
\zeta(x-\beta(x)), &\text{ if }2x>\alpha(x)+\beta(x).
\end{cases}
\end{equation}

\begin{figure}[htp]
	\centering
	\begin{tikzpicture}[thick, scale=1]
		\draw[thick,->] (0,0) -- (0,1.2) ;
		\foreach \t in {0, 1.618, 2.618, 4.236, -1.618, -2.618, -4.236}
		{
		\draw (\t,-0.05) -- (\t,0.05);			
		\draw[domain=(-0.25+\t):(0.25+\t)] plot (\x, {(-64*(\x-\t)*(\x-\t)+(160/27))/7});
		\draw[domain=(0.25+\t):(1/3+\t)] plot (\x, {((64/27)*(3*(\x-\t)-1)*(3*(\x-\t)-1)*(96*(\x-\t)-11))/7});
		\draw[domain=(-1/3+\t):(-0.25+\t)] plot (\x, {((64/27)*(-3*(\x-\t)-1)*(-3*(\x-\t)-1)*(-96*(\x-\t)-11))/7});
	}
        \draw[thick,red] (-1.618,0) -- (1.618,0);
        \draw[thick,red] (2.618,0) -- (4.236,0);
        \draw[thick,red] (-4.236,0) -- (-2.618,0);
        \draw[thick,blue] (1.618,0) -- (2.618,0);
        \draw[thick,blue] (-5,0) -- (-4.236,0);
        \draw[thick,blue] (-2.618,0) -- (-1.618,0);
        \draw[thick,blue,->] (4.236,0) -- (5,0);
        \node  at (-1.618-1.618-1,-1/5) {$-2\tau-1$};
        \node  at (-1.618-1,-1/5) {$-\tau-1$};
        \node  at (-1.618,-1/5) {$-\tau$};
        \node  at (0,-1/5) {$0$};
        \node  at (1.618,-1/5) {$\tau$};
        \node  at (1.618+1,-1/5) {$\tau+1$};
        \node  at (1.618+1+1.618,-1/5) {$2\tau+1$};
	\end{tikzpicture}
\caption{The graph of $V$}
\end{figure}

\begin{lemma}\label{lemma: s-equivariant}
	The potential $V$ is $S$-equivariant with range $1$.
\end{lemma}

\begin{proof}
	Fix $x,y\in\R$. if  \[(S-x) \cap B_{1}(0)=(S-y) \cap B_{1}(0),\] then \[S\cap B_1(x)-x=S\cap B_1(y)-y.\] If $\alpha(x)-x>-1$, then $\alpha(x)\in S\cap B_1(x)$. Then $\alpha(y)\in S\cap B_1(y)$ and $\alpha(x)-x=\alpha(y)-y$.   
	
	If $\alpha(x)-x\leq-1$, then $\beta(x)-\alpha(x)>x-\alpha(x)\geq1$. Since $\alpha(x), \beta(x)\in S$, we have $\beta(x)-\alpha(x)=\tau$. Similarly, we also have $\beta(y)-\alpha(y)=\tau$. Thus    $\beta(x)-x=\alpha(x)-x+\tau\leq\tau-1<1$. Then $\beta(x)-x=\beta(y)-y$. Hence $\alpha(x)-x=(\beta(x)-x)-(\beta(x)-\alpha(x))=(\beta(y)-y)-(\beta(y)-\alpha(y))=\alpha(y)-y$. 
	
	Hence we always have $\alpha(x)-x=\alpha(y)-y$. Similarly, $\beta(x)-x=\beta(y)-y$. Therefore, by the definition of the function $V$, we have $V(x)=V(y)$. 

\end{proof} 

\begin{lemma}
	The potential $V$ is not periodic.
\end{lemma}
 
\begin{proof}
	 The potential $V(x)$ reaches the maximum $160/27$ if and only if $x\in S$. If the potential $V$ had a positive period $T$, then we must have $S+T\subset S$, which is impossible since $S$ is non-periodic. 
\end{proof}
 
So far, we have constructed the  interaction function $H$ in (\ref{eq:interaction}) where the potential $V$ is given by (\ref{eq:potentialV}).

\section{The explicit one-dimensional example and its multiple solutions}\label{solutions}
Our goal in this section is to present the one-dimensional Fibonacci Frenkel-Kontorova model, i.e. the interaction function $H$ given in the above section, for which we can obtain minimal configurations, equilibrium configurations and multiple equilibria with or without any rotation number.

\subsection{Minimal configurations}\label{sec : minimal configurations}
In this section, we aim to find a minimal configuration with rotation number $(3\tau+1)/2$ where such minimal configurations exist according to \cite{GGP2006}. Here we will give a more concrete construction.  

\subsubsection{Local shapes of the Fibonacci quasi-crystal}
For any $x\in\R$, the translation  $S-x$ of $S$ is still a quasi-crystal. Let $S+\R=\{S-x\mid x\in\R\}$ be the collection of all translations of $S$.

	Let $P$ be a given patch of quasi-crystal $S$ and $U$ be a subset in $\R$. The cylinder set $\Omega_{P,U}$ is the set of all quasi-crystals in $S+\R$ that contain a copy of $P$ translated by an element of $U$, that is, $$\Omega_{P, U}:=\{\T\in S+\R \mid P-u \text { is a }  \text {patch in $\T$ for some } u \in U\}.$$
	In particular, we denote $\Omega_P:=\Omega_{P,\{0\}}$. It is easy to check that $\Omega_{P, U}=\bigcup_{u\in U}(\Omega_P-u)$.

For any integer $l\geq 1$,  let
\begin{align*}
c_{l}&=S\cap[-\tau^{2l},\tau^{2l}],\\
\varepsilon_{l,1}&=S\cap[-\tau^{2l},\tau^{2l+2}],\\
\varepsilon_{l,2}&=S\cap[\tau^{2l-1},2\tau^{2l+2}]-\tau^{2l+1}.
\end{align*}

We also use the corresponding symbolic words to represent these patches, for example:

	\begin{align*}
	c_{1}&=ba|ab,\\
	\varepsilon_{1,1}&=ba|abaab,\\
	\varepsilon_{1,2}&=ba|ababaab;
	\end{align*}
	\begin{align*}
	c_{2}&=ababa|abaab,\\
	\varepsilon_{2,1}&=ababa|abaababaabaab,\\
	\varepsilon_{2,2}&=ababa|abaababaababaabaab.
	\end{align*}
	Let $C_l:=\Omega_{c_l}$, $\mathcal{E}_{l,1}:=\Omega_{\varepsilon_{l,1}}$ and $\mathcal{E}_{l,2}:=\Omega_{\varepsilon_{l,2}}$. Obviously, $\mathcal{E}_{l,1}$ and $\mathcal{E}_{l,2}$ are subsets of $C_l$.
	\begin{remark}
		The definitions of patches $c_l,\varepsilon_{l,1}$ and  $\varepsilon_{l,2}$ above are motivated as follows. 
		Consider the function $$\mathcal{L}^l: C_l\rightarrow \R\quad\T\mapsto\inf\{t>0\mid \T-t\in	C_l\}.$$ 
		The patches translated from $c_l$ are distributed over the real line and we use  the function $\mathcal{L}^l$ to measure the distance between every pair of adjacent patches. In the following Lemma~\ref{lemma : partition of hull}, we will show that the distance can only take two values $\tau^{2l+1}$ and $\tau^{2l+2}$, i.e., the range of $\mathcal{L}^l$ is $\{\tau^{2l+1},\tau^{2l+2}\}$. Then, it is easy to see that $\varepsilon_{l,1}$ and $\varepsilon_{l,2}$ are two different  patches which both start and  end with $c_l$ and do not have any other $c_l$ in between.
	\end{remark}
	
	We define the following point set of $\R$:
		\[S^l:=\{x\in\R\mid P-x=c_l \text{ for some patches $P$ in $S$}\}.\] For any $x\in\R$, let $\alpha_l(x)$ be the largest number in $S^l$ such that $\alpha_l(x)\leq x$ and $\beta_l(x)$ be the least number in $S^l$ such that $\beta_l(x)>x$.
	
	\begin{lemma}\label{lemma : partition of hull}
		For each $l\geq1$,
		\begin{compactenum}[(i)]
			\item $C_l=\mathcal{E}_{l,1}\sqcup\mathcal{E}_{l,2}$,
			\item $S+\R=\Omega_{\varepsilon_{l,1},[0,\tau^{2l+1})}\sqcup\Omega_{\varepsilon_{l,2},[0,\tau^{2l+2})}$.
		\end{compactenum}	
	\end{lemma}
	
	\begin{proof}
		(1). Firstly, we show that the range of $\mathcal{L}^l$ is $\{\tau^{2l+1},\tau^{2l+2}\}$. The Fibonacci quasi-crystal $S$ is the corresponding geometric word of the symbolic word $\omega$ in section~\ref{sec:fibonaccichain}. Then the substitution rule $\rho^2$ defined on any subword of $\omega$ can be also defined on any patches of $S$ by the connection between geometric words and symbolic words. Since $\rho^2(w)=w$, the image of a patch of $S$ under $\rho^2$ is still a patch of $S$. In particular, we have $$\rho^2\left(S\bigcap[-\tau^{2l}, \tau^{2l}]\right)=S\bigcap[-\tau^{2l}\times \tau^2, \tau^{2l}\times \tau^2]=S\bigcap[-\tau^{2(l+1)}, \tau^{2(l+1)}],$$ that is, $c_{l+1}=\rho^2(c_l)$. Hence $\mathcal{L}^{l+1}(\T)=\tau^{2}\mathcal{L}^{l}(\T)$ for all $\T\in C_{l+1}$. We only need to show the claim when $l=1$. Notice that $\varepsilon_{1,1}=ba|abaab=c_1aab=ba|ac_1$, we have $\mathcal{L}^{1}(\mathcal{E}_{1,1})=\{\tau^3\}$. It is also easy to see that $\mathcal{L}^{1}(\mathcal{E}_{1,2})=\{\tau^4\}$. We claim that  \begin{align*}
      &\left[(0,\tau^3)\bigcup (\tau^3,\tau^4)\bigcup (\tau^4,+\infty)\right]\\ &\quad \quad\quad \quad\quad \quad\quad \quad\bigcap \left\{t>0\mid \T\in C_l,  \T-t\in	C_l\text{ and }\T-s\not\in C_l \,\forall s\in(0,t)\right\}=\emptyset.
    \end{align*} In fact, for any $t\in(0, \tau^3) $ and any $\mathcal{T}\in C_l$, if $\mathcal{T}-t\in C_l$, then the  corresponding symbolic word of $\mathcal{T}-t$ must be $\cdots ba|ab\cdots$. Then the corresponding symbolic word of $\mathcal{T}$ could only be $\cdots baab|\cdots$ or $\cdots baa|b\cdots$, which is impossible since $\mathcal{T}\in C_l$. For $t\in(\tau^3,\tau^4)$, in the same way,  the corresponding symbolic word of $\T$ could only be $\cdots baab\cdot\cdot|\cdots$ or $\cdots baab\cdots|\cdots$. Furthermore, the word $\cdots baab\cdots|\cdots$ must be $\cdots baabaaa|\cdots$, which is impossible since  $aaa$ is not a subword of $w$. Hence the corresponding symbolic word of $\T$ could only be $\cdots baab\cdot\cdot|\cdots$. Since $\T\in C_l$, its corresponding symbolic word is $\cdots baabba|ab\cdots$, which is impossible since  $bb$ is not a subword of $w$.  If $t\in (\tau^4, +\infty)$, the corresponding symbolic word of $\T$ could only be $\cdots baabPba|ab\cdots$.  the subword $P$ formed by more than one letters must be like $ababa\cdots aba$ since $bb$ and $aa$ are not subwords of $P$. In fact, if $aa$ is a subword of $P$, then $baab$ is a subword of $\cdots P\cdots$, which contradicts that $\T-s\not\in C_l \,\forall s\in(0,t)$. Then the word $\cdots baabPba|ab\cdots$ has subword $ababab$ which is impossible since $aaa$ is not a subword of $\omega$ and  $\rho(aaa)=ababab$. Hence the range of $\mathcal{L}^1$ is $\{\tau^3,\tau^4\}$. It is obvious that $$C_1=(\mathcal{L}^1)^{-1}(\tau^3)\sqcup(\mathcal{L}^1)^{-1}(\tau^4)=\mathcal{E}_{1,1}\sqcup\mathcal{E}_{1,2}.$$ Since $\varepsilon_{l+1,i}=\rho^2(\varepsilon_{l,i})$, $i\in\{1,2\}$, by induction, we get (i).%

		(2). Recall the definition of $\alpha_l$ and $\beta_l$. We have proved that $\beta_l(x)-\alpha_l(x)$ is $\tau^{2l+1} \text{ or } \tau^{2l+2}$. If $\beta_l(x)-\alpha_l(x)=\tau^{2l+1} $, then
		\begin{equation*}\label{eq: superword-1}
		(S-\alpha_l(x))\cap[-\tau^{2l},\tau^{2l+2}]=\varepsilon_{l,1}.
		\end{equation*} If $\beta_l(x)-\alpha_l(x)=\tau^{2l+2} $, then
		\begin{equation*}\label{eq: superword-2}
		(S-\alpha_l(x))\cap[-\tau^{2l},\tau^{2l+2}+\tau^{2l}]=\varepsilon_{l,2}.
		\end{equation*}

		For any $\T=S-x\in S+\R$, we have $\T-(x-\alpha_l(x))=S-\alpha_l(x)$ and thus $$S+\R\subset\Omega_{\varepsilon_{l,1},[0,\tau^{2l+1})}\cup\Omega_{\varepsilon_{l,2},[0,\tau^{2l+2})}.$$ 
		The proof of the opposite inclusion is obvious and so we complete the proof of (ii).
	\end{proof}

\subsubsection{The frequencies of the local shapes}
If $\beta_l(x)-\alpha_l(x)=\tau^{2l+1}$, then we denote the corresponding symbolic word of  $S\cap[\alpha_l(x),\beta_l(x)]$ by $A_l$. If $\beta_l(x)-\alpha_l(x)=\tau^{2l+2}$, then we denote the corresponding symbolic word of  $S\cap[\alpha_l(x),\beta_l(x)]$ by $B_l$. For example, \[A_1=aba,\quad B_1=ababa,\quad A_2=abaababa,\quad B_2=abaababaababa.\]
Notice that $A_2=A_1B_1$ and $B_2=A_1B_1B_1$. Since $\varepsilon_{l+1,i}=\rho^2(\varepsilon_{l,i})$ for $i\in\{1,2\}$,  we have $A_{l+1}=\rho^2(A_l)$ and $B_{l+1}=\rho^2(B_l)$. Hence we have $A_{l+1}=A_l B_l$ and $B_{l+1}=A_l B_l B_l$ for all $l\geq 1$ by induction. Let  \[M:=\begin{pmatrix}
1 & 1\\
1 & 2\\
\end{pmatrix}:=\begin{pmatrix}
\text{the number of $A_l$ in } A_{l+1} & \text{the number of $B_l$ in } A_{l+1}\\
\text{the number of $A_l$ in } B_{l+1} & \text{the number of $B_l$ in } B_{l+1}\\
\end{pmatrix}\] then 
\begin{align*}
  M^n&=\frac{1}{\sqrt{5}}\begin{pmatrix}
    \tau^{-2n+1}+\tau^{2n-1} & -\tau^{-2n}+\tau^{2n}\\
    -\tau^{-2n}+\tau^{2n} & \tau^{-2n-1}+\tau^{2n+1}
    \end{pmatrix}\\ 
    &=\begin{pmatrix}
      \text{the number of $A_l$ in } A_{l+n} & \text{the number of $B_l$ in } A_{l+n}\\
      \text{the number of $A_l$ in } B_{l+n} & \text{the number of $B_l$ in } B_{l+n}\\
      \end{pmatrix}.
\end{align*}
Consider the two limits \[\lim_{n\to\infty}\frac{\text{the number of $A_l$ in } A_{l+n}}{| A_{l+n}|}=\lim_{n\to\infty}\frac{\text{the number of $A_l$ in } B_{l+n}}{|B_{l+n}|}=\frac{1}{\sqrt{5}\tau^{2l+2}},\]
where $|\cdot|$ is the assignment function defined in section \ref{sec:fibonaccichain}.
Since the symbolic word of  $S\cap[-\tau^4,\tau^3]$ is $B_1|A_1$, the absolute frequency of $A_l$ at $0$ is the limit $1/(\sqrt{5}\tau^{2l+2})$ above. In the same way, the absolute frequency of $B_l$ at $0$ is $1/(\sqrt{5}\tau^{2l+1})$.

\subsubsection{The construction of a minimal configuration}\label{constru of minimal}
For any integer $l\geq 1$, let $\B_l$ be an oriented  1-dimensional branched manifold in $\R^2$, which consists of two circles $\gamma_{l,1}$ and $\gamma_{l,2}$ that are tangent to one another  at the tangent point $R_l$. In fact, this branched manifold is the Anderson-Putnam complex \cite{Sadun2008}. The circumferences of $\gamma_{l,1}$ and $\gamma_{l,2}$ are $\tau^{2l+1}$ and $\tau^{2l+2}$, respectively. Given two points $\xi$ and $\eta$ on the same circle, the oriented length of the arc from $\xi$ to $\eta$ is denoted by $d(\xi,\eta)$. Let $m_{l,i}(x)$ denote the point on $\gamma_{l,i}$ such that $d(R_l,m_{l,i}(x))=x$, $i\in\{1,2\}$.  
\begin{figure}[htp]
    \centering
\def\svgwidth{0.25\columnwidth}
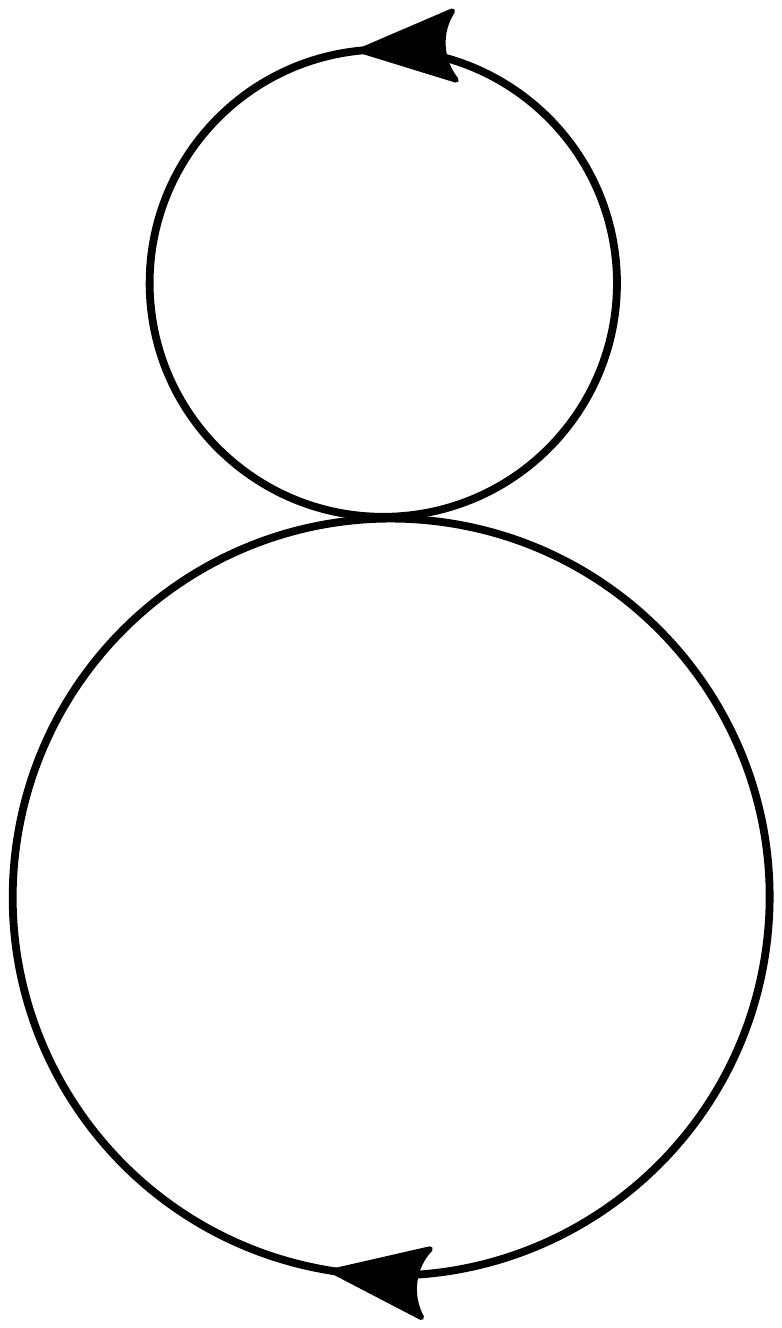

    \caption{ The branched manifold $\mathcal{B}_l$}
    \label{fig:branchedManifold}
\end{figure}

For any $l\geq1$, we define the map $\kappa_l:\B_{l+1}\rightarrow\B_l$ which is illustrated by

\begin{figure}[htp]
    \centering
\def\svgwidth{0.45\columnwidth}
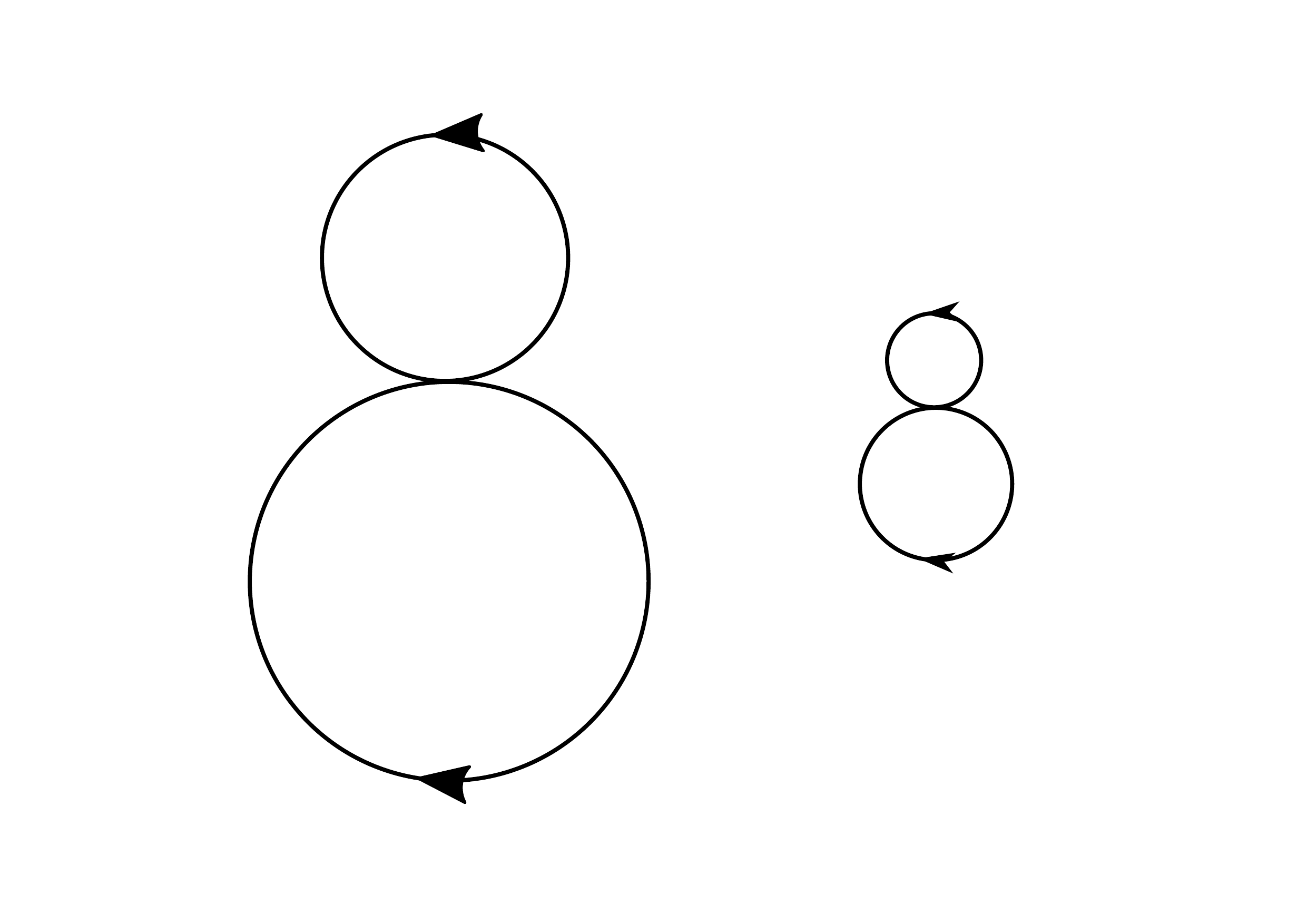

\def\svgwidth{0.45\columnwidth}
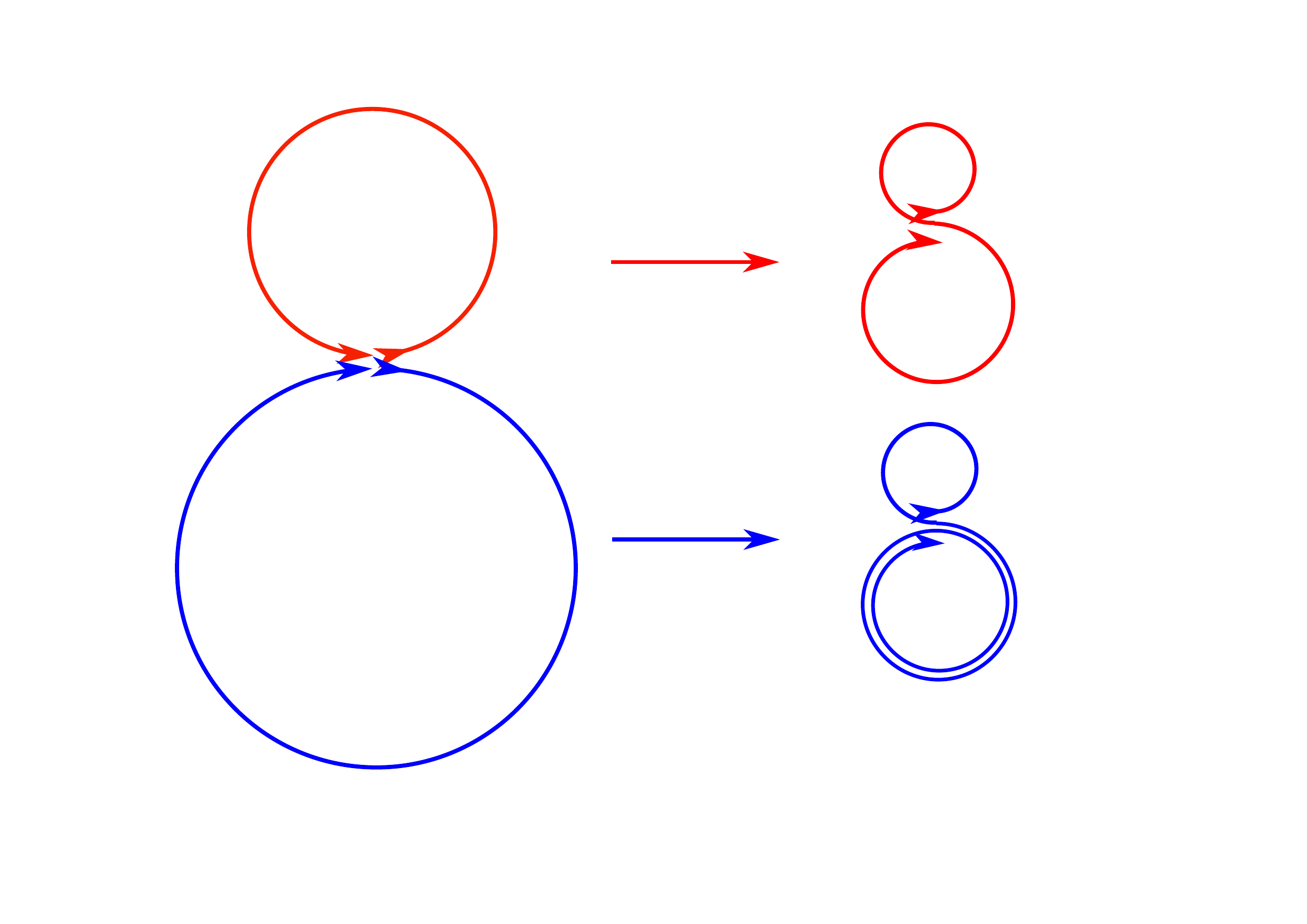

    \caption{The map $\kappa_l$}
    \label{fig:kappaBlack}
\end{figure}

and the projection $\pi_l: \R\rightarrow \B_l$ is defined by \[\pi_l(x)=\begin{cases}
	m_{l,1}(x-\alpha_l(x)), &\text{ if } \beta_l(x)-\alpha_l(x)=\tau^{2l+1};\\
	m_{l,2}(x-\alpha_l(x)), &\text{ if } \beta_l(x)-\alpha_l(x)=\tau^{2l+2}.
	\end{cases}\]

On the one hand, since $A_{l+1}=A_l B_l$ and $B_{l+1}=A_l B_l B_l$ for all $l\geq 1$, we have $\kappa_l\circ\pi_{l+1}=\pi_l$. 
On the other hand, since the projection $\pi_l$ is a \emph{ covering map} from $\R$ to $\B_{l}$ and $\pi_l(S^l)=\{R_l\}$, the preimage $(\pi_l)^{-1}(y)$ of each point $y$ is a point set in $\R$, and $S\cap \overline{B_1(x)}$ is the same patch up to translations for any point $x$ in $(\pi_l)^{-1}(y)$. Therefore, the value of $V$ is the same on $(\pi_l)^{-1}(y)$ by Lemma \ref{lemma: s-equivariant}. Such property of $V$ allows us to define the potential on $\B_l$ by $\widehat{V}:=V\circ(\pi_l)^{-1}:\B_l\rightarrow\R$  and the function $\widehat{H}:\gamma_{l,1}\times\gamma_{l,1}\cup\gamma_{l,2}\times\gamma_{l,2}\rightarrow\R$ which maps $(\xi,\eta)$ to $\frac{1}{2}d(\xi,\eta)^2+\widehat{V}(\xi).$ 

Now we construct the minimal configuration.
	
\emph{Step 1}: Fix $l=1$. For each $i\in\{1,2\}$, let $b_{1,i}$ be a point of $\gamma_{1,i}\backslash \{R_1\}$, and as in section~\ref{VP}  we extend $\widehat{H}$ acting on the triple segment $(R_1, b_{1,i}, R_1)$ and obtain:
	\begin{align*}
 \widehat{H}(R_1, b_{1,i}, R_1)&=\widehat{H}(R_1,b_{1,i})+\widehat{H}(b_{1,i},R_1)\\ 
 &=\frac{1}{2}d(R_1,b_{1,i})^2+\widehat{V}(R_1)+\frac{1}{2}d(b_{1,i},R_1)^2+\widehat{V}(b_{1,i})\\
 &=\frac{1}{2}d(R_1,b_{1,i})^2+\widehat{V}(R_1)+\frac{1}{2}(\tau^{2+i}-d(R_1,b_{1,i}))^2+\widehat{V}(b_{1,i})\\
 &=d(R_1,b_{1,i})^2-\tau^{2+i}d(R_1,b_{1,i})+V(d(R_1,b_{1,i}))+\tau^{4+2i}/2+V(0). 
 \end{align*}

One can easily see that $\widehat{H}(R_1, b_{1,i}, R_1)$ reaches its minimum at $b_{1,i}^*:=m_{1,i}(\tau^{2+i})$ which is the antipodal point of $R_1$ on $\gamma_{1,i}$ since $d(R_1,b_{1,i}^*)=\tau^{2+i}/2$ and the nonnegative potential $V$ vanishes at $\tau^{2+i}/2$.
We can then construct a bi-infinite increasing sequence $(\theta_{1,n})_{n\in\Z}$ of $\R$ with $\theta_{1,0}=0$ such that
\[
(\theta_{1,n})_{n\in\Z} = (\pi_1)^{-1}(\{R_1,b_{1,1}^*, b_{1,2}^*\}).
\]
This is because the pre-image of $\{R_1,b_{1,1}^*, b_{1,2}^*\}$ under $\pi_1$ is a discrete accountable subset of $\R$.
We thus obtain a configuration  $(\theta_{1,n})_{n\in\Z}$ on $\R$ which is minimal  on each segment $[\alpha_1(x),\beta_1(x)]$.

\emph{Step 2:}
For any $l\geq 1$ and each $i\in\{1,2\}$, we first define the number $N_{l,i}$ by iteration  
\[
\begin{pmatrix}
	N_{l+1,1}\\
	N_{l+1,2}
	\end{pmatrix}=\begin{pmatrix}
	1&1\\
	1&2
	\end{pmatrix} \begin{pmatrix}
	N_{l,1}\\
	N_{l,2}
	\end{pmatrix},
	\qquad \text{with $N_{1,1}=N_{1,2}:=2$. }
	\] 
Note that $(N_{l,1}, N_{l,2})=(2f_{2l-2},2f_{2l-1})$, where $\{f_i\}$ is the Fibonacci number defined in section \ref{sec:fibonaccichain}.

We now extend $\widehat{H}$ acting on the segment $(R_l, b_{l,i,1}, \dots, b_{l,i,N_{l,i}-1}, R_l)$
where  $b_{l,i,1},$ $ b_{l,i,2},\dots, b_{l,i,N_{l,i}-1}$ are $N_{l,i}-1$ different points in $\gamma_{l,i}\backslash\{R_l\}$, and we require that the subscript $j$ of the point $b_{l,i,j}$ increases along the orientation of the circle $\gamma_{l,i}$. 
Then $\widehat{H}(R_l, b_{l,i,1}, \dots, b_{l,i,N_{l,i}-1}, R_l)$ reaches its minimum at $(b_{l,i,1}, \dots, b_{l,i,N_{l,i}-1} )= (b_{l,i,1}^*, \dots, b_{l,i,N_{l,i}-1}^*)$.

Similarly as before, since pre-image $(\pi_l)^{-1}(\{R_l,b_{l,i,j}^*\mid 1\leq j\leq N_{l,i}-1\})$ is a discrete accountable subset of $\R$, we can obtain a bi-infinite increasing sequence $(\theta_{l,n})_{n\in\Z}$  on $\R$ with $\theta_{l,0}=0$ which is  minimal on each segment $[\alpha_l(x),\beta_l(x)]$.

\emph{Step 3}: By the two steps above we get a sequence of configurations  $(\theta^m)_{m\in \mathbb{Z}^+} = ((\theta_{m,n})_{n\in\Z})_{m\in\mathbb{Z}^+}$ (see Figure~5 for numerical simulations of equilibrium configurations with $m=5$). In the subsequent two sections, we will show that this sequence is in a compact subset of $\R^\Z$ and its accumulation point with respect to $m\rightarrow + \infty$ is a minimal configuration with rotation number $(3\tau+1)/2$.
	\begin{figure}  \label{Fig:minimal configurations}
		\centering  \includegraphics[scale=0.3]{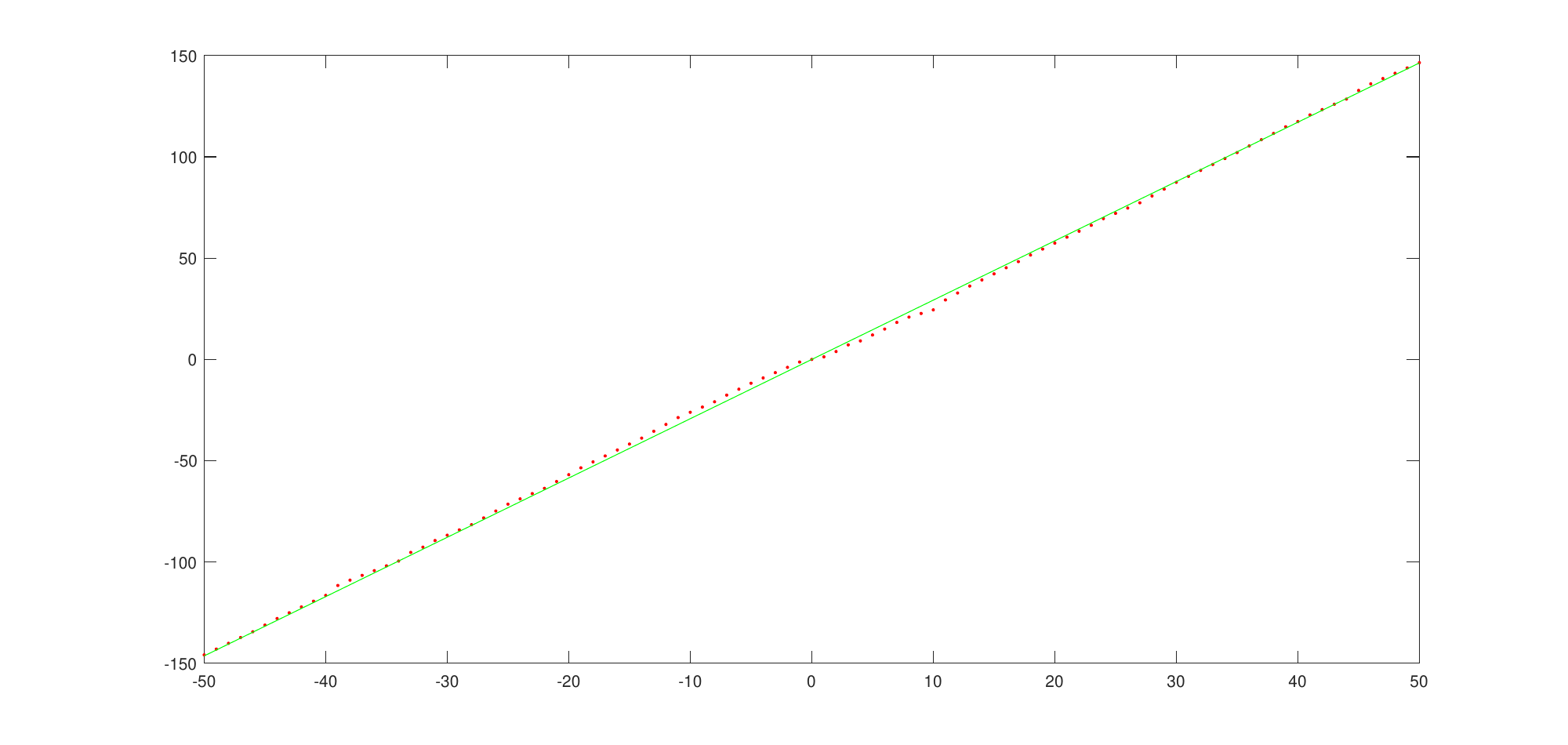}  
		\caption{The red points are the graph of $(\theta_{5,n})_{n\in[-50,50]}$ and the green line is $y=(3\tau+1)/2x$.}   
	\end{figure}

\subsubsection{ Combinatorics of minimal configurations}	

		\begin{proposition}[\cite{BKR2006}]\label{prop: the  combinatorics}
		Let $(\theta_1, \dots,\theta_n)$ be a minimal segment and let $I$ be an interval in $[\theta_1,\theta_n]$, then there exists an integer $n\in \Z^{+}\cup\{0\}$ such that for any pair of disjoint intervals $I_1=I+u_1$ and $I_2=I+u_2$ in $[\theta_1,\theta_n]$ which satisfy that for each $\theta$ in $I$ and $k = 1, 2$:
  \[ S\cap B_{1}(\theta)+u_{k}= S\cap B_{1}\left(\theta+u_{k}\right)\]
  the cardinality $\operatorname{Card}(I_k\cap (\theta_1, \dots,\theta_n))\in \{n,n+1,n+2\} $ for $k=1,2$.
	\end{proposition}


 \begin{corollary}\label{coro: the  combinatorics}
 For each $l\geq 1$ and $i=1,2$, let $u\in \R^\Z$ be a minimal configuration(resp. let $(u_p, \dots, u_q)$ be a minimal segment), then for any two connected component $I_1, I_2$ of $\pi_{l}^{-1}(\gamma_{l,i})$(resp. any two connected component $I_1, I_2$ of $\pi_{l}^{-1}(\gamma_{l,i})$ which does not intersect $(-\infty, u_p]\cup[u_q, +\infty)$), we have \[\left|\operatorname{Card}(u\bigcap I_1)-\operatorname{Card}(u\bigcap I_2)\right|\leq 2\](resp.\[\left|\operatorname{Card}((u_p, \dots, u_q)\bigcap I_1)-\operatorname{Card}((u_p, \dots, u_q)\bigcap I_2)\right|\leq 2).\] 
\end{corollary}

\subsubsection{The proof of item (i) of Theorem~\ref{maintheorem}}
	\begin{lemma}
		For each $l\geq 1$, $(\theta_{l,n})_{n\in\Z}$ has rotation number $(3\tau+1)/2$.
	\end{lemma}
	
	\begin{proof}
	Let $n_{l,i}$ be the number of times $\pi_l([\theta_{l,0},\theta_{l,n}])$ covers completely the circle $\gamma_{l,i}$. Then \[n_{l,1}\tau^{2l+1}+n_{l,2}\tau^{2l+2}\leq\theta_{l,n}-\theta_{l,0}\leq n_{l,1}\tau^{2l+1}+n_{l,2}\tau^{2l+2}+2\tau^{2l+2}\] and \[n_{l,1}N_{l,1}+n_{l,2}N_{l,2}\leq n\leq n_{l,1}N_{l,1}+n_{l,2}N_{l,2}+2N_{l,2}.\] Thus \[\frac{n_{l,1}\tau^{2l+1}+n_{l,2}\tau^{2l+2}}{n_{l,1}N_{l,1}+n_{l,2}N_{l,2}+2N_{l,2}}\leq \frac{\theta_{l,n}}{n}\leq\frac{n_{l,1}\tau^{2l+1}+n_{l,2}\tau^{2l+2}+2\tau^{2l+2}}{n_{l,1}N_{l,1}+n_{l,2}N_{l,2}}\]
	Then the rotation number $\rho_l$ of $(\theta_{l,n})_{n\in\Z}$ is the limit (if exists): \[\lim_{n\to+\infty}\frac{n_{l,1}\tau^{2l+1}+n_{l,2}\tau^{2l+2}}{n_{l,1}N_{l,1}+n_{l,2}N_{l,2}}.\]   When $n$ goes to $+\infty$ the quantity \[\frac{n_{l,i}}{n_{l,1}\tau^{2l+1}+n_{l,2}\tau^{2l+2}}\] goes to the absolute  frequency of $A_{l}$ if $i=1$ or of $B_{l}$ if $i=2$. Hence \[\rho_l=\frac{1}{\operatorname{Freq}_0(A_l)N_{l,1}+\operatorname{Freq}_0(B_l)N_{l,2}}=\frac{3\tau+1}{2}. 
	\]
	\end{proof}
	
\begin{lemma}\label{lemma: compact conf}
	There exists $M>0$ such that for any $l\geq 1$, $n\in\Z$, we have \[\theta_{l,n+1}-\theta_{l,n}\leq M.\]
\end{lemma}

\begin{proof}
Let $M(m)=2|\B_m|=2\tau^{2m+3}$, where $m\geq 2$. Suppose by contradiction that there exist $l(m)$ and $n(m)$ such that \[\theta_{l(m),n(m)+1}-\theta_{l(m),n(m)}> M(m).\] Then there exists minimal segment \[(\theta_{l(m),n_1}, \dots, \theta_{l(m),n(m)},\theta_{l(m),n(m)+1}, \dots, \theta_{l(m),n_2})\] such that 
   \[\gamma_{m,i}\subset \pi_m([\theta_{l(m),n(m)},\theta_{l(m),n(m)+1}]) \text{ for some } i.\]
 Let $n_{l(m),i}$ be the number of times $\pi_{m-1}([\theta_{l(m),0},\theta_{l(m),n}])$ covers completely the circle $\gamma_{m-1,i}$. Since $\kappa_{m-1}(\gamma_{m,i})=\B_{m-1}$, let $\kappa_{m-1}$ act on each side then we have \[\B_{m-1}\subset \pi_{m-1}([\theta_{l(m),n(m)},\theta_{l(m),n(m)+1}]).\] 
  By Corollary \ref{coro: the  combinatorics}, for each connected component $I_{m-1,i}$ of $\pi_{m-1}^{-1}(\gamma_{m-1,i})$, we have \[\left|\operatorname{Card}(\theta_{l(m)}\cap I_{m-1,i} )\right|\leq 2.\] 
  Since $\theta_{l(m)}$ has rotation number $(3\tau+1)/2$, for all $m\geq 2$, we have  \begin{align*}
   (3\tau+1)/2&=\lim_{n\to+\infty}\frac{\theta_{l(m),n}}{n}\\ 
   &\geq \lim_{n\to+\infty}\frac{n_{l(m),1}\tau^{2m-1}+n_{l(m),2}\tau^{2m}}{2(n_{l(m),1}+1)+2(n_{l(m),2}+1)}\\ 
   &=\tau^{2m-1} \lim_{n\to\infty}\frac{n_{l(m),1}+\tau n_{l(m),2}}{2n_{l(m),1}+2 n_{l(m),2}+4} .
  \end{align*}
 Notice that $\lim_{n\to\infty}\frac{n_{l(m),1}+\tau n_{l(m),2}}{2n_{l(m),1}+2 n_{l(m),2}+4}$ equals either $1/2$ or $\tau/2$ and so the right side of the inequality goes to $+\infty$ when $m\rightarrow +\infty$, which is impossible.
 \end{proof}

By Lemma~\ref{lemma: compact conf}, for each $l\geq1$,  the distance $|\theta_{l,n+1}-\theta_{l,n}|$ is bounded for all $n\in\Z$. Therefore, all the configurations $((\theta_{m,n})_{n\in\Z})_{m\in\mathbb{Z}^+}$ are contained in a compact subset of $\R^\Z$, which guarantees that there exists at least one accumulation point as $m\rightarrow +\infty$. We denote by $(\theta_{\infty,n})_{n\in\Z}$ any of these accumulation points.
	
\begin{theorem}
	The configuration $(\theta_{\infty,n})_{n\in\Z}$ is a minimal configuration with rotation number $(3\tau+1)/2$.
\end{theorem}
	
\begin{proof}
		It is easy to show  that $(\theta_{\infty,n})_{n\in\Z}$ is minimal.  In fact, for any finite segment $(\theta_{\infty,j}, \cdots, \theta_{\infty,k}),  k>j$ which is a limit point of minimal segments of $(\theta_{m,n})_{n=j}^k$ as $m\rightarrow +\infty$, it is straightforward to show that this segment is also minimal.
		
		We just need to show its rotation number is $(3\tau+1)/2$.
	Let $n_{\infty,l,i}$ be the number of times $\pi_l([\theta_{\infty,0},\theta_{\infty,n}])$ covers completely the circle $\gamma_{l,i}$. Then \[n_{\infty,l,1}\tau^{2l+1}+n_{\infty,l,2}\tau^{2l+2}\leq\theta_{\infty,n}-\theta_{\infty,0}\leq n_{\infty,l,1}\tau^{2l+1}+n_{\infty,l,2}\tau^{2l+2}+2\tau^{2l+2}\] and by Corollary \ref{coro: the  combinatorics}, we have
	\[n_{\infty,l,1}(N_{l,1}-2)+n_{\infty,l,2}(N_{l,2}-2)\leq n\leq n_{\infty,l,1}(N_{l,1}+2)+n_{\infty,l,2}(N_{l,2}+2)+2(N_{l,2}+2).\] Thus \begin{align*}
    &\frac{n_{\infty,l,1}\tau^{2l+1}+n_{\infty,l,2}\tau^{2l+2}}{n_{\infty,l,1}(N_{l,1}+2)+n_{\infty,l,2}(N_{l,2}+2)+2(N_{l,2}+2)} \\ &\leq \frac{\theta_{\infty,n}}{n}\\ &\leq\frac{n_{\infty,l,1}\tau^{2l+1}+n_{\infty,l,2}\tau^{2l+2}+2\tau^{2l+2}}{n_{\infty,l,1}(N_{l,1}-2)+n_{\infty,l,2}(N_{l,2}-2)}.
  \end{align*} 
	When $n$ goes to $+\infty$, the quantity \[\frac{n_{\infty,l,i}}{n_{\infty,l,1}N_{l,1}+n_{\infty,l,2}N_{l,2}}\] goes to the absolute  frequency of $A_{l}$ if $i=1$ and of $B_{l}$ if $i=2$. Then the rotation number $\rho_\infty$ of $(\theta_{\infty,n})_{n\in\Z}$ satisfies \begin{align*}
    &\frac{1}{\operatorname{Freq}_0(A_l)(N_{l,1}+2)+\operatorname{Freq}_0(B_l)(N_{l,2}+2)}\\ &\leq \rho_\infty \\ &\leq\frac{1}{\operatorname{Freq}_0(A_l)(N_{l,1}-2)+\operatorname{Freq}_0(B_l)(N_{l,2}-2)}
	\qquad \forall ~l\geq 1.
  \end{align*}
	Let $l\rightarrow+\infty$, and we obtain
	\[\rho_\infty=\frac{3\tau+1}{2}.\]
\end{proof}

\subsection{Equilibrium configurations 
close to the anti-integrable limit}\label{existence of equi conf}

In this section, we will first find an equilibrium configuration $(u_i)_{i\in\Z}$ with rotation number $(3\tau+1)/2$ satisfying
\begin{equation}\label{eq:equilibrium-condition}
-(\Delta u)_i+ V'(u_i)=0,
\end{equation}
where $\Delta$ denotes the discrete Laplacian. Then a specific calculation of such configuration will be given.

To start with, let $h:\Z\rightarrow\R$ be defined by
 $$h(i)=\frac{3\tau+1}{2} i$$ 
 and let $g:\Z\rightarrow\R$ be defined by
  $$g(i)=\argmin_{x\in S} |x-h(i)|=\begin{cases}
\alpha\circ h(i) &\text{ if } 2h(i)\leq \alpha\circ h(i)+\beta\circ h(i)\\
\beta\circ h(i) &\text{ if } 2h(i)> \alpha\circ h(i)+\beta\circ h(i)
\end{cases}.$$ 
Since the distance between any pair of adjacent points in $S$ is $1$ or $\tau$, the closed ball with diameter $\tau$ centered at $h(i)$ must contain some point of $S$, which means that $|g(i)-h(i)|\leq\tau/2$ for all $i\in\Z$. 
Now we start to find an equilibrium configuration in our context. First we reduce the existence of equilibria to a fixed point of a contraction mapping on a closed neighborhood of $(g(i))_{i\in\Z}$ with the metric $\delta(u, u^{\prime})=\sup _{i \in \Z}|u_{i}-u_{i}^{\prime}|$. More precisely,  we consider a space $\Pi$ defined by \[\Pi:=\left\{u :|u_{i}-g(i)| \leq \frac{\tau}{62} \text { for all } i \in \mathbb{Z}\right\}\] and a mapping $\Phi$ on $\Pi$ defined by \[ \Phi:\Pi\rightarrow\Pi,\quad u\mapsto \left(-\frac{1}{128}(\Delta u)_i+g(i)\right)_{i\in\Z}.\] Notice that for all $u\in\Pi$,  we have
\begin{align*}
|(\Delta u)_i|&\leq|(\Delta u)_i-(\Delta g)_i|+|(\Delta g)_i|\leq 4\sup_{i \in \Z}|u_{i}-g(i)|+2\tau\leq\frac{64\tau}{31}.
\end{align*}
Thus, we get $|\Phi(u)_{i}-g(i)|=\frac{1}{128}|(\Delta u)_i|\leq\frac{\tau}{62}$, so the mapping $\Phi$ is well defined. Moreover, $\Phi$ is a contraction mapping since \[|\Phi(u)_i-\Phi(u^{\prime})_i|=\frac{1}{128}|(\Delta u)_i-(\Delta u')_i|\leq\frac{1}{32}\sup _{i \in \Z}|u_{i}-u_{i}^{\prime}|.\] Hence by the contraction mapping principle, $\Phi$ has a unique fixed point $u$ satisfying
\begin{equation}\label{eq:fixed-point}
u_i=-\frac{1}{128}(\Delta u)_i+g(i) \qquad \forall~i\in\Z.
\end{equation}
Now it only remains to show that formula (\ref{eq:fixed-point}) and (\ref{eq:equilibrium-condition}) are equivalent.  In fact, the above fixed point $u$  satisfies $\min\{|u_i-x|\mid x\in S\}=|u_i-g(i)|\leq \tau/62\leq 1/4$ for all $i\in \Z$, and hence $V'(u_i)=\frac{\dd}{\dd x}(-64(x-g(i))^2+160/27)|_{x=u_i}=-128(u_i-g(i))$. Then we can obtain (\ref{eq:equilibrium-condition}) by multiplying  by $128$ on  both sides of (\ref{eq:fixed-point}), by which we can conclude that $u$ is an equilibrium configuration. What's more, since $|u_i-g(i)|\leq\frac{\tau}{62}$ and $|g(i)-h(i)|<\tau/2$ hold for all $i\in\Z$, the rotation number of $u$ is  $(3\tau+1)/2$, the same as the slope of $h$.

In conclusion, we have obtained
\begin{theorem}\label{thm:equiconf}
	The configuration $u$ is  an equilibrium configuration with rotation number $\theta$.
\end{theorem}

Next, let's calculate the equilibrium configuration $u$ we constructed.
Let $\alpha=1/128$, then  (\ref{eq:fixed-point}) is equivalent to
\[\alpha u_{i-1}+(1-2\alpha) u_{i}+\alpha u_{i+1}=a_i  \qquad \forall~i\in\Z.
\]
Let $T$ be a tridiagonal operator defined by
\[T e_{i}=\alpha e_{i-1}+(1-2\alpha) e_{i}+\alpha e_{i+1}, \quad i \in \mathbb{Z}\]
where $\{e_i:i\in\Z\}$ is the orthonormal basis of the sequence space. Consider the truncation matrix
\[T_{n} =\begin{pmatrix}
{1-2\alpha} & {\alpha} & {0} & {\cdots} & {0} & {0} \\ {\alpha} & {1-2\alpha} & {\alpha} & {\cdots} & {0} & {0} \\ {0} & {\alpha} & {1-2\alpha} & {\cdots} & {0} & {0} \\ {\vdots} & {\vdots} & {\vdots} & {\ddots} & {\vdots} & {\vdots} \\ {0} & {0} & {0} & {\cdots} & {1-2\alpha} & {\alpha} \\ {0} & {0} & {0} & {\cdots} & {\alpha} & {1-2\alpha}
\end{pmatrix}_{n\times n},\]

which is invertible by  \cite{huang1997analytical}[Theorem 3.1]. And $\left\{\left\|T_{n}^{-1} e_{n}\right\|\right\}$, $\left\{\left\|T_{n}^{*^{-1}} e_{n}\right\|\right\}$ are bounded due to  \cite{BKR2006}[Corollary 6.2]. Then using \cite{antonyselvan2016}[Theorem 3.1], we have
\[u=\lim\limits_{n\to \infty}{T_{2n+1}^{-1}y_n},\] 
where $y_n=(g(-n),\dots,g(0),\dots,g(n))^T$ is the truncation of $(g(i))_{i\in\Z}$.

	\begin{remark}
	\begin{itemize}
	\item [(i)] In fact, the method of finding equilibrium configurations in this article can be applied to all $h$ satisfying $|(\Delta h)_i|<\infty$. For example, let 
	\[h_1(i)=\begin{cases}
	i^2,& \text{if} \, i\geq 0,\\
	-i^2,& \text{if} \, i< 0.
	\end{cases}\] 
	In this case, $\lim\limits_{n\to\pm\infty} h_1(i)/i=\pm\infty$ and hence the equilibrium configuration has no rotation number.
	\item [(ii)] See Figure~6 for numerical simulations of equilibrium configurations with or without rotation numbers.
	\end{itemize}
%
	
	\begin{figure}[htp]
  \centering
\def\svgwidth{0.45\columnwidth}
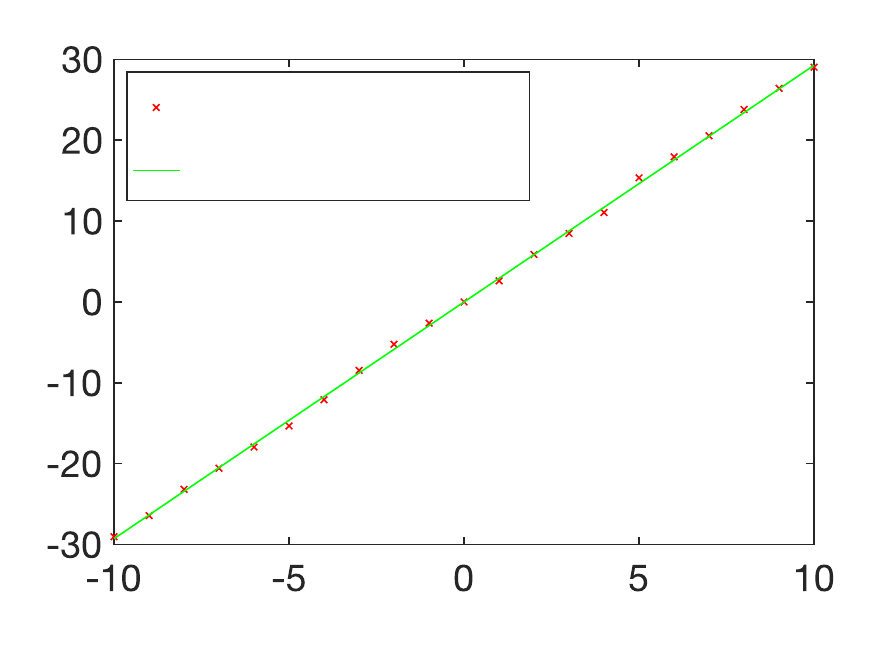

\def\svgwidth{0.45\columnwidth}
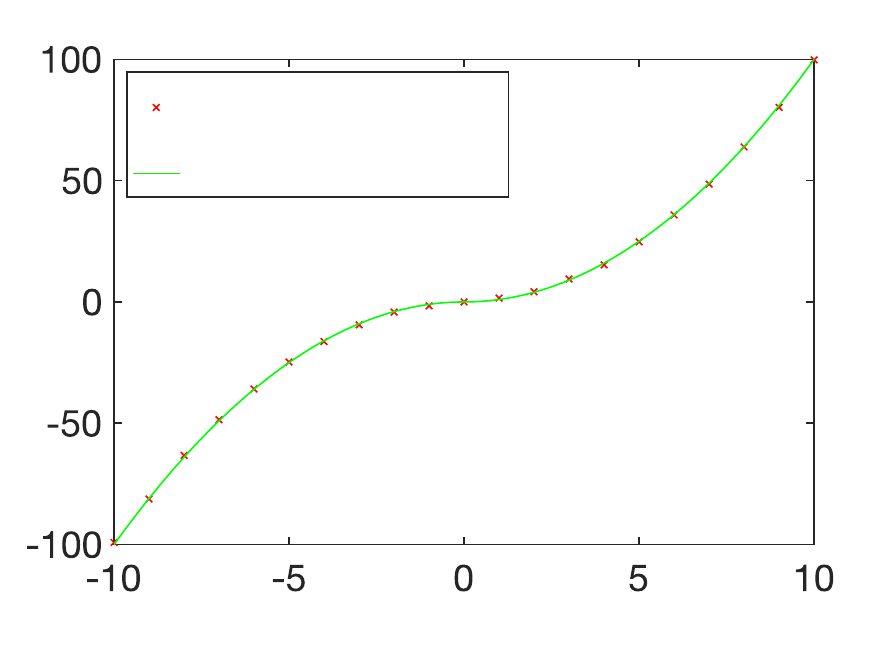

  \caption{The configuration $T_{2n+1}^{-1}y_{n}$ for $h$ and $h_1$  when $n=500$}
  \label{fig:equiconf}
 \end{figure}

\end{remark}

The following theorem is a special case of  \cite{Trevino2019}[Theorem~1].
\begin{theorem}\label{thm: existence of equi-conf}
	Let $h:\Z\rightarrow\R$ satisfy $|(\Delta h)_i|<\infty\text{ for all } i\in\Z.$ Then there exists a $\lambda_0$ such that for any $\lambda>\lambda_0$, there exists an equilibrium configuration $u$ of type $h$ with respect to \begin{equation}\label{eq:hlambda}
		H_{\lambda}(\xi,\eta)=\frac{1}{2}(\xi-\eta)^2+\lambda V(\xi).
	\end{equation}
\end{theorem}

\begin{remark}
The original paper \cite{Aubry1990} considers the notion of anti-integrable limits for the traditional periodic Frenkel-Kontorova models and obtains the same types of equilibrium configurations. The authors also show the chaotic properties of these ``exotic" equilibrium configurations.
\end{remark}

Moreover, in our context, we could show that there  also exist non-minimal equilibrium configurations:
\begin{theorem}\label{thm: nonminimal}
	Let $h:\Z\rightarrow\R$ satisfy $|(\Delta h)_i|<\infty\text{ for all } i\in\Z.$ Then there exists $\lambda_0, \lambda_1\in \mathbb{R}$ satisfying $\lambda_1>\lambda_0$ such that for any $\lambda>\lambda_0$, there exists  an equilibrium configuration $u$ of type $h$ with respect to $H_{\lambda}(\xi,\eta).$
	In particular, if $\lambda>\lambda_1$, the equilibrium configuration obtained above is non-minimal.
\end{theorem}

To show Theorem \ref{thm: nonminimal}, we just need a lemma:

\begin{lemma}\label{lemma: nonminimal}
	For any $\lambda>-4/\zeta''(0)$, if $u$ is an equilibrium configuration with respect to $H_{\lambda}$	and each component $u_i$ lies in the quadratic part of $V$  (that is, for each $u_i$, there exists a neighborhood $U$ of $u_i$ such that $V|_{U}$ is a quadratic function on $U$), then $u$ is non-minimal.
\end{lemma}

\begin{remark}
	For the function $V$ defined in Section \ref{EP}, it is easy to see that a point $x$ lies in the quadratic part  of $V$ if and only if $\min_{y\in S}|x-y|<1/4$, where $S$ is the Fibonacci quasicrystal. 
\end{remark}

\begin{proof}
	Fix $\lambda$ and $u$ that meet the conditions. Since each component $u_i$ lies in the quadratic part of $V$, we can regard the sequence $g(i)$ as the closest local maximum  point of $V$ to $u_i$. Notice that 
	\begin{equation}\label{eq: equi_config_lambda}
	-(\Delta u)_i+ \lambda V'(u_i)=0
	\end{equation}
	and $V'(g(i))=0$ for all $i\in\Z$, then there exists some index $i_0$ such that $u_{i_0}\ne g(i_0)$. Otherwise $(\Delta u)_i=\lambda V'(u_i)=\lambda V'(g(i))=0$ and thus $u_{i+1}-u_i$ is a positive constant for all $i\in\Z$. Since each $u_i=g(i)$ as a local maximum  point of $V$ is contained in the Fibonacci chain $S$, and the distance between any two points in $S$ has the form $m\cdot\tau+n\cdot 1$, we suppose that $u_{i+1}-u_i=m\cdot \tau+n\cdot 1$ for some nonnegative integers $m$ and $n$. That is, the corresponding symbolic word of the geometric word $S\cap [u_i,u_{i+1}]$ contains $m$ $a$-letters  and $n$ $b$-letters. Then  the frequency of the letter $a$ in $w$ is equal to $m/(m+n)$, which is rational. However, the frequency of the letter $a$ in $S$ is $1/\tau$, which was calculated in the proof of Lemma \ref{lemma: quasi-crystal}. Without loss of generality, suppose that $i_0=0$. In order to show that $u$ is non-minimal, we show that $(u_{-1}, u_0, u_1)$ is  non-minimal. That is, we find a $u_0':=u_0'(\lambda,u)$ such that
	\begin{equation}\label{ieq: varia_condition}
	H_{\lambda}(u_{-1},u_0')+H_{\lambda}(u_0', u_1)<H_{\lambda}(u_{-1},u_0)+H_{\lambda}(u_0,u_1).
	\end{equation}
	Before we give the value of $u_0'$, there are some relations among $u_0, g(0)$ and $\bar{u}:=(u_{-1}+u_1)/2$.
	Since $u_0$ lies in the support of the quadratic part of $V$, from Taylor's formula and the condition of $\lambda$, we have
	\begin{equation}\label{ieq: deriv_condition}
    \begin{aligned}
      \lambda V'(u_0)&=\lambda V''(g(0))(u_0-g(0))\\ 
  &=\lambda \zeta''(0)(u_0-g(0))\begin{cases}
	<-4(u_0-g(0))&\text{if } u_0>g(0);\\
	>-4(u_0-g(0))&\text{if } u_0<g(0).
	\end{cases}
    \end{aligned}
	\end{equation}
	Then from (\ref{eq: equi_config_lambda}) and (\ref{ieq: deriv_condition}), we know that 
	\begin{equation}\label{ieq: order1}
	\bar{u}=u_0+\frac{\lambda V'(u_0)}{2}\begin{cases}
	<-u_0+2g(0)&\text{if } u_0>g(0);\\
	>-u_0+2g(0)&\text{if } u_0<g(0).
	\end{cases}
	\end{equation}	
	That is, 
	\begin{equation}\label{ieq: order2}
	|u_0-g(0)|<|\bar{u}-g(0)|.
	\end{equation}
	And from (\ref{ieq: order1}) we can see that the sign of $u_0-g(0)$ is different from the sign of $\bar{u}-g(0)$, thus we have \begin{equation}\label{ieq: order3}
	|\bar{u}-g(0)|<|\bar{u}-u_0|.
	\end{equation}
	Now, we discuss three different cases and give the value of $u_0'$ in each case. 
	\begin{description}
		\item [Case 1.] If $|\bar{u}-g(0)|\leq 1/2$, where $1/2$ is half of the minimal distance between two adjacent local maximum points of $V$. 
		Let $u_0'=\bar{u}$. Consider the function $F(x)=\frac{1}{2}(u_{-1}-x)^2+\frac{1}{2}(x-u_1)^2$. $F(x)$ is a quadratic function and takes minimum at $\bar{u}$. Hence, 		\begin{equation}\label{ieq: convex1}
		\frac{1}{2}(u_{-1}-u_0')^2+\frac{1}{2}(u_0'-u_1)^2\leq\frac{1}{2}(u_{-1}-u_0)^2+\frac{1}{2}(u_0-u_1)^2.
		\end{equation} 
		
		Since $|\bar{u}-g(0)|\leq 1/2$ and $V$ is 
		increasing on the interval $[g(0)-1/2, g(0)]$,  decreasing on the interval $[g(0), g(0)+1/2]$, and especially, strictly monotone on its quadratic parts, by (\ref{ieq: order2}), we have 
		\begin{equation}\label{ieq: potential1}
		V(u_0')< V(u_0).
		\end{equation}
		Multiply both sides of the inequality (\ref{ieq: potential1}) by $\lambda$ and add the inequality (\ref{ieq: convex1}), then we have (\ref{ieq: varia_condition}).
		
		\item [Case 2.] If $\bar{u}-g(0)> 1/2$. Let $u_0'=g(0)+ 1/3$, where $1/3$ is the radius of each bump of $V$. Then \begin{equation}\label{ieq: potential2}
		V(u_0')=0\leq V(u_0).
		\end{equation}
		By (\ref{ieq: order3}), we have $$|\bar{u}-u_0'|=|\bar{u}-g(0)-1/3|<|\bar{u}-g(0)|<|\bar{u}-u_0|.$$ Then using the monotonicity of $F(x)$, we have \begin{equation}\label{ieq: convex2}
		\frac{1}{2}(u_{-1}-u_0')^2+\frac{1}{2}(u_0'-u_1)^2<\frac{1}{2}(u_{-1}-u_0)^2+\frac{1}{2}(u_0-u_1)^2.
		\end{equation} 
		Multiply both sides of the inequality (\ref{ieq: potential2}) by $\lambda$ and add the inequality (\ref{ieq: convex2}), then we have (\ref{ieq: varia_condition}).
		
		\item [Case 3.] If $\bar{u}-g(0)< -1/2$, let $u_0'=g(0)- 1/3$. The rest of the proof is the same as the one in Case 2. 
	\end{description}
\end{proof}

{{\vspace{-0.5em}\noindent\bf Proof of Theorem \ref{thm: nonminimal}}\quad}
	For any $h$ with $|(\Delta h)_i|<\infty\text{ for all } i\in\Z$, Theorem \ref{thm: existence of equi-conf} ensures the existence of equilibrium configuration of type $h$ with respect to $H_\lambda$ for any $\lambda$ larger than some $\lambda_0$. Let $\lambda_{\frac{1}{2}}=\max\{\lambda_0, -4/\zeta''(0)\}$. Then for any $\lambda>\lambda_{\frac{1}{2}}$, the equilibrium configuration $u$ whose existence is proved in Theorem \ref{thm: existence of equi-conf} is non-minimal by Lemma \ref{lemma: nonminimal}, as long as we have that each $u_i$ lies in the quadratic part of $V$. 

	To solve the problem of whether each $u_i$ can lie in the quadratic part of $V$, let us see the proof of Theorem \ref{thm: existence of equi-conf} carefully. In our example of $V$, a sufficient condition is that the radius of the space that the contraction mapping is smaller than $1/4$. A possible construction of the contraction is  \[\Pi_\lambda=\left\{u :|u_{i}-g(i)| \leq \frac{2\max\{|a|,|b|\}+\sup_{i \in \Z}|(\Delta h)_i|}{-\lambda\zeta''(0)-4} \text { for all } i \in \mathbb{Z}\right\},\]\[ \Phi_\lambda:\Pi_\lambda\rightarrow\Pi_\lambda,\quad u\mapsto \left(\frac{1}{\lambda\zeta''(0)}(\Delta u)_i+g(i)\right)_{i\in\Z}.\]
	For all $u\in\Pi_\lambda$, since
	\begin{align*}
	|(\Delta u)_i|&\leq|(\Delta u)_i-(\Delta g)_i|+|(\Delta g)_i-(\Delta h)_i|+|(\Delta h)_i|\\
	&\leq 4\sup_{i \in \Z}|u_{i}-g(i)|+4\sup_{i \in \Z}|g({i})-h(i)|+\sup_{i \in \Z}|(\Delta h)_i|\\
	&\leq 4\cdot\frac{2\max\{|a|,|b|\}+\sup_{i \in \Z}|(\Delta h)_i|}{-\lambda\zeta''(0)-4}+4\cdot\frac{\max\{|a|,|b|\}}{2}+\sup_{i \in \Z}|(\Delta h)_i|\\
	&=-\lambda\zeta''(0)\cdot\frac{2\max\{|a|,|b|\}+\sup_{i \in \Z}|(\Delta h)_i|}{-\lambda\zeta''(0)-4},
	\end{align*}
	we get \begin{align*}
	|\Phi_\lambda(u)_{i}-g(i)|=\left|\frac{1}{\lambda\zeta''(0)}(\Delta u)_i\right|\leq\frac{2\max\{|a|,|b|\}+\sup_{i \in \Z}|(\Delta h)_i|}{-\lambda\zeta''(0)-4},
	\end{align*}
	which means that $\Phi_\lambda$ is well-defined. $\Phi_\lambda$ is a contraction mapping because
	\[|\Phi_\lambda(u)_i-\Phi_\lambda(u^{\prime})_i|=\frac{1}{-\lambda\zeta''(0)}|(\Delta u)_i-(\Delta u')_i|\leq\frac{4}{-\lambda\zeta''(0)}\sup _{i \in \Z}|u_{i}-u_{i}^{\prime}|.\]
	Let the radius of $\Pi_\lambda$ be smaller than $1/4$, and we get a restriction on $\lambda$:
	\[\lambda>4\cdot\frac{2\max\{|a|,|b|\}+\sup_{i \in \Z}|(\Delta h)_i|+1}{-\zeta''(0)}\]
	Hence, let $$\lambda_1=\max\left\{\lambda_\frac{1}{2}, 4\cdot\frac{2\max\{|a|,|b|\}+\sup_{i \in \Z}|(\Delta h)_i|+1}{-\zeta''(0)}\right\},$$
	and then for any $\lambda>\lambda_1$, there exists a non-minimal equilibrium configuration of type $h$ with respect to $H_\lambda$.
  {\hfill $\square$\par}

{{\vspace{0.5em}\noindent\bf Proof of item (ii) of Theorem \ref{maintheorem}}\quad}
By Theorem \ref{thm:equiconf}, to show the item (ii) of Theorem \ref{maintheorem}, it suffices to show that $u$ is non-minimal with respect to $H_\lambda$ in (\ref{eq:hlambda})  when $\lambda=1$.

	In fact, by the definition of the function $\zeta$, we have $-4/\zeta''(0)=1/32<1$.  Due to Theorem~\ref{thm:equiconf}, we have that $u$ is an equilibrium configuration with respect to $H_{1}=H$. Since $u\in \Pi=\left\{u :|u_{i}-g(i)| \leq \frac{\tau}{62} \text { for all } i \in \mathbb{Z}\right\}$, we know that each component $u_i$ lies in the quadratic part of $V$. Using Lemma \ref{lemma: nonminimal}, $u$ is non-minimal. 
  {\hfill $\square$\par}

\subsection*{Acknowledgements}
The authors would like to thank the anonymous referees for the valuable comments and suggestions on the manuscript.
We also thank Prof. R. de la Llave for a very careful reading of the manuscript and many suggestions which helped improve a lot the presentation and exposition. X. Su is supported by the National Natural Science Foundation of China (Grant No. 11971060, 11871242).

\bibliographystyle{unsrt}  
\bibliography{references}

\medskip
Received xxxx 20xx; revised xxxx 20xx.
\medskip

\end{document}

%% file: figures/branchedManifold.pdf_tex
\begingroup%
  \makeatletter%
  \providecommand\color[2][]{%
    \errmessage{(Inkscape) Color is used for the text in Inkscape, but the package 'color.sty' is not loaded}%
    \renewcommand\color[2][]{}%
  }%
  \providecommand\transparent[1]{%
    \errmessage{(Inkscape) Transparency is used (non-zero) for the text in Inkscape, but the package 'transparent.sty' is not loaded}%
    \renewcommand\transparent[1]{}%
  }%
  \providecommand\rotatebox[2]{#2}%
  \newcommand*\fsize{\dimexpr\f@size pt\relax}%
  \newcommand*\lineheight[1]{\fontsize{\fsize}{#1\fsize}\selectfont}%
  \ifx\svgwidth\undefined%
    \setlength{\unitlength}{595.27559055bp}%
    \ifx\svgscale\undefined%
      \relax%
    \else%
      \setlength{\unitlength}{\unitlength * \real{\svgscale}}%
    \fi%
  \else%
    \setlength{\unitlength}{\svgwidth}%
  \fi%
  \global\let\svgwidth\undefined%
  \global\let\svgscale\undefined%
  \makeatother%
  \begin{picture}(1,1.41428571)%
    \lineheight{1}%
    \setlength\tabcolsep{0pt}%
    \put(0,0){\includegraphics[width=\unitlength,page=1]{branchedManifold.pdf}}%
    \put(0.47505537,0.81314537){\color[rgb]{0,0,0}\makebox(0,0)[lt]{\lineheight{1.25}\smash{\begin{tabular}[t]{l}$R_l$\end{tabular}}}}%
    \put(0,0){\includegraphics[width=\unitlength,page=2]{branchedManifold.pdf}}%
    \put(0.70821657,1.06735531){\color[rgb]{0,0,0}\makebox(0,0)[lt]{\lineheight{1.25}\smash{\begin{tabular}[t]{l}$m_{l,1}(x)$\end{tabular}}}}%
    \put(0.66514929,0.93815366){\color[rgb]{0,0,0}\makebox(0,0)[lt]{\lineheight{1.25}\smash{\begin{tabular}[t]{l}{\color{red}$x$ }\end{tabular}}}}%
    \put(0.18662462,1.21091272){\color[rgb]{0,0,0}\makebox(0,0)[lt]{\lineheight{1.25}\smash{\begin{tabular}[t]{l}$\gamma_{l,1}$\end{tabular}}}}%
    \put(0.15312787,0.30410847){\color[rgb]{0,0,0}\makebox(0,0)[lt]{\lineheight{1.25}\smash{\begin{tabular}[t]{l}$\gamma_{l,2}$\end{tabular}}}}%
    \put(0.83024038,0.60079368){\color[rgb]{0,0,0}\makebox(0,0)[lt]{\lineheight{1.25}\smash{\begin{tabular}[t]{l}$m_{l,y}(y)$\end{tabular}}}}%
    \put(0.75128385,0.78502575){\color[rgb]{0,0,0}\makebox(0,0)[lt]{\lineheight{1.25}\smash{\begin{tabular}[t]{l}{\color{blue} $y$}\end{tabular}}}}%
  \end{picture}%
\endgroup%

%% file: figures/kappaBlack.pdf_tex
\begingroup%
  \makeatletter%
  \providecommand\color[2][]{%
    \errmessage{(Inkscape) Color is used for the text in Inkscape, but the package 'color.sty' is not loaded}%
    \renewcommand\color[2][]{}%
  }%
  \providecommand\transparent[1]{%
    \errmessage{(Inkscape) Transparency is used (non-zero) for the text in Inkscape, but the package 'transparent.sty' is not loaded}%
    \renewcommand\transparent[1]{}%
  }%
  \providecommand\rotatebox[2]{#2}%
  \newcommand*\fsize{\dimexpr\f@size pt\relax}%
  \newcommand*\lineheight[1]{\fontsize{\fsize}{#1\fsize}\selectfont}%
  \ifx\svgwidth\undefined%
    \setlength{\unitlength}{1190.5511811bp}%
    \ifx\svgscale\undefined%
      \relax%
    \else%
      \setlength{\unitlength}{\unitlength * \real{\svgscale}}%
    \fi%
  \else%
    \setlength{\unitlength}{\svgwidth}%
  \fi%
  \global\let\svgwidth\undefined%
  \global\let\svgscale\undefined%
  \makeatother%
  \begin{picture}(1,0.70714286)%
    \lineheight{1}%
    \setlength\tabcolsep{0pt}%
    \put(0,0){\includegraphics[width=\unitlength,page=1]{kappaBlack.pdf}}%
    \put(0.33117908,0.04838438){\color[rgb]{0,0,0}\makebox(0,0)[lt]{\lineheight{1.25}\smash{\begin{tabular}[t]{l}$\mathcal{B}_{l+1}$\end{tabular}}}}%
    \put(0.69655609,0.22117346){\color[rgb]{0,0,0}\makebox(0,0)[lt]{\lineheight{1.25}\smash{\begin{tabular}[t]{l}$\mathcal{B}_{l}$\end{tabular}}}}%
    \put(0,0){\includegraphics[width=\unitlength,page=2]{kappaBlack.pdf}}%
    \put(0.56156464,0.42276079){\color[rgb]{0,0,0}\makebox(0,0)[lt]{\lineheight{1.25}\smash{\begin{tabular}[t]{l}$\kappa_l$\end{tabular}}}}%
  \end{picture}%
\endgroup%

%% file: figures/kappa.pdf_tex
\begingroup%
  \makeatletter%
  \providecommand\color[2][]{%
    \errmessage{(Inkscape) Color is used for the text in Inkscape, but the package 'color.sty' is not loaded}%
    \renewcommand\color[2][]{}%
  }%
  \providecommand\transparent[1]{%
    \errmessage{(Inkscape) Transparency is used (non-zero) for the text in Inkscape, but the package 'transparent.sty' is not loaded}%
    \renewcommand\transparent[1]{}%
  }%
  \providecommand\rotatebox[2]{#2}%
  \newcommand*\fsize{\dimexpr\f@size pt\relax}%
  \newcommand*\lineheight[1]{\fontsize{\fsize}{#1\fsize}\selectfont}%
  \ifx\svgwidth\undefined%
    \setlength{\unitlength}{1190.5511811bp}%
    \ifx\svgscale\undefined%
      \relax%
    \else%
      \setlength{\unitlength}{\unitlength * \real{\svgscale}}%
    \fi%
  \else%
    \setlength{\unitlength}{\svgwidth}%
  \fi%
  \global\let\svgwidth\undefined%
  \global\let\svgscale\undefined%
  \makeatother%
  \begin{picture}(1,0.70714286)%
    \lineheight{1}%
    \setlength\tabcolsep{0pt}%
    \put(0,0){\includegraphics[width=\unitlength,page=1]{kappa.pdf}}%
    \put(0.47360927,0.54702584){\color[rgb]{0,0,0}\makebox(0,0)[lt]{\lineheight{1.25}\smash{\begin{tabular}[t]{l}{\color{red}$\kappa_l|_{\gamma_{l,1}}$}\end{tabular}}}}%
    \put(0.47698015,0.33803101){\color[rgb]{0,0,0}\makebox(0,0)[lt]{\lineheight{1.25}\smash{\begin{tabular}[t]{l}{\color{blue}$\kappa_l|_{\gamma_{l,2}}$}\end{tabular}}}}%
  \end{picture}%
\endgroup%

%% file: figures/equiconf1.pdf_tex
\begingroup%
  \makeatletter%
  \providecommand\color[2][]{%
    \errmessage{(Inkscape) Color is used for the text in Inkscape, but the package 'color.sty' is not loaded}%
    \renewcommand\color[2][]{}%
  }%
  \providecommand\transparent[1]{%
    \errmessage{(Inkscape) Transparency is used (non-zero) for the text in Inkscape, but the package 'transparent.sty' is not loaded}%
    \renewcommand\transparent[1]{}%
  }%
  \providecommand\rotatebox[2]{#2}%
  \newcommand*\fsize{\dimexpr\f@size pt\relax}%
  \newcommand*\lineheight[1]{\fontsize{\fsize}{#1\fsize}\selectfont}%
  \ifx\svgwidth\undefined%
    \setlength{\unitlength}{420bp}%
    \ifx\svgscale\undefined%
      \relax%
    \else%
      \setlength{\unitlength}{\unitlength * \real{\svgscale}}%
    \fi%
  \else%
    \setlength{\unitlength}{\svgwidth}%
  \fi%
  \global\let\svgwidth\undefined%
  \global\let\svgscale\undefined%
  \makeatother%
  \begin{picture}(1,0.75)%
    \lineheight{1}%
    \setlength\tabcolsep{0pt}%
    \put(0,0){\includegraphics[width=\unitlength,page=1]{equiconf1.pdf}}%
    \put(0.50484314,0.03154264){\color[rgb]{0,0,0}\makebox(0,0)[lt]{\lineheight{1.25}\smash{\begin{tabular}[t]{l}$x$\end{tabular}}}}%
    \put(0.03407439,0.37831486){\color[rgb]{0,0,0}\rotatebox{89.874486}{\makebox(0,0)[lt]{\lineheight{1.25}\smash{\begin{tabular}[t]{l}$y$\end{tabular}}}}}%
    \put(0.2267768,0.61716187){\color[rgb]{0,0,0}\makebox(0,0)[lt]{\lineheight{1.25}\smash{\begin{tabular}[t]{l} {\footnotesize $T^{-1}_{1001}y_{500}$}\end{tabular}}}}%
    \put(0.2272938,0.54921559){\color[rgb]{0,0,0}\makebox(0,0)[lt]{\lineheight{1.25}\smash{\begin{tabular}[t]{l}{\footnotesize $y=(3\tau+1)/2x$}\end{tabular}}}}%
  \end{picture}%
\endgroup%

%% file: figures/equiconf2.pdf_tex
\begingroup%
  \makeatletter%
  \providecommand\color[2][]{%
    \errmessage{(Inkscape) Color is used for the text in Inkscape, but the package 'color.sty' is not loaded}%
    \renewcommand\color[2][]{}%
  }%
  \providecommand\transparent[1]{%
    \errmessage{(Inkscape) Transparency is used (non-zero) for the text in Inkscape, but the package 'transparent.sty' is not loaded}%
    \renewcommand\transparent[1]{}%
  }%
  \providecommand\rotatebox[2]{#2}%
  \newcommand*\fsize{\dimexpr\f@size pt\relax}%
  \newcommand*\lineheight[1]{\fontsize{\fsize}{#1\fsize}\selectfont}%
  \ifx\svgwidth\undefined%
    \setlength{\unitlength}{420bp}%
    \ifx\svgscale\undefined%
      \relax%
    \else%
      \setlength{\unitlength}{\unitlength * \real{\svgscale}}%
    \fi%
  \else%
    \setlength{\unitlength}{\svgwidth}%
  \fi%
  \global\let\svgwidth\undefined%
  \global\let\svgscale\undefined%
  \makeatother%
  \begin{picture}(1,0.75)%
    \lineheight{1}%
    \setlength\tabcolsep{0pt}%
    \put(0,0){\includegraphics[width=\unitlength,page=1]{equiconf2.pdf}}%
    \put(0.50988235,0.03046706){\color[rgb]{0,0,0}\makebox(0,0)[lt]{\lineheight{1.25}\smash{\begin{tabular}[t]{l} $x$\end{tabular}}}}%
    \put(0.03687865,0.37704689){\color[rgb]{0,0,0}\rotatebox{89.665507}{\makebox(0,0)[lt]{\lineheight{1.25}\smash{\begin{tabular}[t]{l}$y$\end{tabular}}}}}%
    \put(0.23387805,0.61965771){\color[rgb]{0,0,0}\makebox(0,0)[lt]{\lineheight{1.25}\smash{\begin{tabular}[t]{l}{\footnotesize $T_{1001}^{-1}y_{500}$}\end{tabular}}}}%
    \put(0.23400882,0.54852627){\color[rgb]{0,0,0}\makebox(0,0)[lt]{\lineheight{1.25}\smash{\begin{tabular}[t]{l}{\footnotesize$y=\mathrm{sign}(x)x^2$}\end{tabular}}}}%
  \end{picture}%
\endgroup%

%% file: Template_1.bbl
\begin{thebibliography}{10}

\bibitem{BK2004}
O.~M. Braun and Y.~S. Kivshar.
\newblock {\em The {F}renkel-{K}ontorova model}.
\newblock Texts and Monographs in Physics. Springer-Verlag, Berlin, 2004.

\bibitem{Sadun2008}
Lorenzo Sadun.
\newblock {\em Topology of tiling spaces}, volume~46 of {\em University Lecture
  Series}.
\newblock American Mathematical Society, Providence, RI, 2008.

\bibitem{SuL2012}
Xifeng Su and Rafael de~la Llave.
\newblock K{AM} {T}heory for {Q}uasi-periodic {E}quilibria in
  {O}ne-{D}imensional {Q}uasi-periodic {M}edia.
\newblock {\em SIAM J. Math. Anal.}, 44(6):3901--3927, 2012.

\bibitem{SuL2017}
Xifeng Su and Rafael de~la Llave.
\newblock A continuous family of equilibria in ferromagnetic media are ground
  states.
\newblock {\em Comm. Math. Phys.}, 354(2):459--475, 2017.

\bibitem{Lions2003}
Pierre-Louis Lions and Panagiotis~E. Souganidis.
\newblock Correctors for the homogenization of {H}amilton-{J}acobi equations in
  the stationary ergodic setting.
\newblock {\em Comm. Pure Appl. Math.}, 56(10):1501--1524, 2003.

\bibitem{Aubry1990}
Serge Aubry and Gilles Abramovici.
\newblock Chaotic trajectories in the standard map. {T}he concept of
  anti-integrability.
\newblock {\em Phys. D}, 43(2-3):199--219, 1990.

\bibitem{Trevino2019}
Rodrigo Trevi\~{n}o.
\newblock Equilibrium configurations for generalized {F}renkel-{K}ontorova
  models on quasicrystals.
\newblock {\em Comm. Math. Phys.}, 371(1):1--17, 2019.

\bibitem{GGP2006}
Jean-Marc Gambaudo, Pierre Guiraud, and Samuel Petite.
\newblock Minimal configurations for the {F}renkel-{K}ontorova model on a
  quasicrystal.
\newblock {\em Communications in Mathematical Physics}, 265(1):165--188, 2006.

\bibitem{GPT2017}
Eduardo Garibaldi, Samuel Petite, and Philippe Thieullen.
\newblock Calibrated configurations for {F}renkel-{K}ontorova type models in
  almost periodic environments.
\newblock {\em Ann. Henri Poincar\'{e}}, 18(9):2905--2943, 2017.

\bibitem{Llave2013}
Timothy Blass and Rafael de~la Llave.
\newblock The analyticity breakdown for {F}renkel-{K}ontorova models in
  quasi-periodic media: numerical explorations.
\newblock {\em J. Stat. Phys.}, 150(6):1183--1200, 2013.

\bibitem{AB2010}
Boris Adamczewski and Yann Bugeaud.
\newblock Transcendence and {D}iophantine approximation.
\newblock In {\em Combinatorics, automata and number theory}, volume 135 of
  {\em Encyclopedia Math. Appl.}, pages 410--451. Cambridge Univ. Press,
  Cambridge, 2010.

\bibitem{BKR2006}
R~Balasubramanian, SH~Kulkarni, and R~Radha.
\newblock Solution of a tridiagonal operator equation.
\newblock {\em Linear algebra and its applications}, 414(1):389--405, 2006.

\bibitem{huang1997analytical}
Y~Huang and WF~McColl.
\newblock Analytical inversion of general tridiagonal matrices.
\newblock {\em Journal of Physics A: Mathematical and General}, 30(22):7919,
  1997.

\bibitem{antonyselvan2016}
A.~Antony Selvan and R.~Radha.
\newblock Invertibility of a tridiagonal operator with an application to a
  non-uniform sampling problem.
\newblock {\em Linear Multilinear Algebra}, 65(5):973--990, 2017.

\end{thebibliography}
